\documentclass[12pt,reqno]{amsart}
\usepackage{amsfonts}
\usepackage{amssymb}
\usepackage{amsthm}
\usepackage{amsmath}
\usepackage{arabtex}

\theoremstyle{plain}
\newtheorem {Lem}{Lemma}
\newtheorem {The}{Theorem}

\usepackage{color}

\theoremstyle{remark}

\theoremstyle{definition}
\newtheorem {lemma}{Lemma}[section]
\newtheorem {fragment}[lemma]{}

\newtheorem {prob}{Problem}
\newtheorem {deff}[lemma]{Definition}

\newif\ifcomm
\let\ifcomm\iffalse

\newcommand\supb[1]{{}^{\fboxsep1pt\fbox{$\scriptstyle\mskip1mu#1\mskip-1mu$}}}

\def\Sp{\operatorname{Sp}}
\def\Ep{\operatorname{Ep}}

\def\SL{\operatorname{SL}}
\def\GL{\operatorname{GL}}
\def\GSp{\operatorname{GSp}}

\def\Max{\operatorname{Max}}
\def\Ann{\operatorname{Ann}}
\def\Hom{\operatorname{Hom}}
\def\Rad{\operatorname{Rad}}
\def\EO{\operatorname{EO}}

\def\diag{\operatorname{diag}}
\def\sic{\text{\rm sc}}
\def\ad{\text{\rm ad}}
\def\map{\longrightarrow}
\def\bar{\overline}
\def\rk{\operatorname{rk}}
\def\epsilon{\varepsilon}
\def\e{\varepsilon}
\def\a{\alpha}
\newcommand{\al}{\alpha}
\def\b{\beta}
\def\g{\gamma}

\def\Nat{{\Bbb N}}
\def\Int{{\Bbb Z}}
\newcommand{\ma}{\mathfrak{a}}
\newcommand{\mb}{\mathfrak{b}}
\newcommand{\mc}{\mathfrak{c}}
\newcommand{\md}{\mathfrak{d}}
\newcommand{\mm}{\mathfrak{m}}
\newcommand{\mq}{\mathfrak{q}}
\def\A{\operatorname{A}}
\def\B{\operatorname{B}}
\def\C{\operatorname{C}}
\def\D{\operatorname{D}}
\def\F{\operatorname{F}}
\def\G{\operatorname{G}}
\def\E{\operatorname{E}}

\title[Relative Commutator Calculus]{Relative commutator calculus\\
in Chevalley groups}

\author{Roozbeh Hazrat}
\address{School of Computing, Engineering and Mathematics, University of Western Sydney, Australia}
\email{r.hazrat@uws.edu.au}
\author{Nikolai Vavilov}
\address{Department of Mathematics and Mechanics,
St.-Petersburg State University, St,-Petersburg, Russia}
\email{nikolai-vavilov@yandex.ru}
\author{Zuhong Zhang}
\address{Department of  Mathematics, Beijing Institute of Technology, Beijing, China}
\email{zuhong@gmail.com}
\begin{thanks}
{The work of the second author was supported by
RFFI projects 09-01-00762, 09-01-00784, 09-01-00878, 09-01-91333,
09-01-90304, 10-01-90016, 10-01-92651, 11-01-00756.
At the final stage his work was
supported by the state financed task project 6.38.74.2011 ``Structure
theory and geometry of algebraic groups, and their applications in
representation theory and algebraic $K$-theory''  at the
Saint Petersburg State University.
The third author acknowledges the support of NSFC grant 10971011 and
the support from Beijing Institute of Technology. }
\end{thanks}

\begin{document}

\begin{abstract}
We revisit localisation and patching method in the setting of
Chevalley groups. Introducing certain subgroups of relative elementary
Chevalley groups, we develop relative versions of the conjugation
calculus and the commutator calculus in Chevalley groups
$G(\Phi,R)$, $\rk(\Phi)\geq 2$, which are
both more general, and substantially easier than the ones available
in the literature. For classical groups such relative commutator
calculus has been recently developed by the authors in \cite{RZ,RNZ}.
As an application we prove the mixed commutator formula,
\[ \big [E(\Phi,R,\ma),G(\Phi,R,\mb)\big ]=\big [E(\Phi,R,\ma),E(\Phi,R,\mb)\big], \]
for two ideals $\ma,\mb\unlhd R$. This answers a problem posed in a
paper by Alexei Stepanov and the second author.}

\end{abstract}

\maketitle

\setfarsi
\novocalize

\begin{arabtext}
br rhgo_daram hzAr jAa dAm nhy\\ 
\, gUeyey keh begeyramt agar gAm nhy \\ 
\, xeyAm
\end{arabtext}

\medskip

\begin{flushright}
{\it O Life, you put thousand traps in my way\\
Dare to try, is what you clearly say\\}
Omar Khayam
\end{flushright}




\section{Introduction}

One of the most powerful ideas in the study of groups of points
of reductive groups over rings is localisation. It allows to reduce
many important problems over arbitrary commutative rings, to
similar problems for semi-local rings. Localisation comes in a
number of versions. The two most familiar ones are {\bf localisation
and patching}, proposed by Daniel Quillen \cite{Qu}
and Andrei Suslin \cite{Sus},
and {\bf localisation--completion}, proposed by Anthony Bak \cite{B4}.
\par
Originally, the above papers addressed the case of the general linear
group $\GL(n,R)$. Soon thereafter, Suslin himself, Vyacheslav Kopeiko,
Marat Tulenbaev, Giovanni Taddei, Leonid Vaserstein, Li Fuan,
Eiichi Abe, You Hong, and others
proposed working versions of localisation and patching for other
classical groups, such as symplectic and orthogonal ones, as well
as exceptional Chevalley groups, see, for example,
\cite{kopeiko,taddei,V1,Li1,Li2,VY} and further references in
\cite{NV91,BV3,SV,RN}. Recently, these methods were further
generalised to unitary groups, and isotropic reductive groups,
by Tony Bak, Alexei Stepanov, ourselves, Victor Petrov,
Anastasia Stavrova, Ravi Rao, Rabeya Basu, and others
\cite{BBR,Basu,BRN,BS,BV2,BV3,BRK,RH,RH2,RN1,P1,petrov2,petrov3,PS08,
stavrova,SV10}
\par
As a matter of fact, both methods rely on a large body of common
calculations, and technical facts, known as {\bf conjugation calculus}
and {\bf commutator calculus}. Their objective is to obtain explicit
estimates of the modulus of continuity in $s$-adic topology for
conjugation by a specific matrix, in terms of the powers of $s$
occuring in the denominators of its entries, and similar estimates
for commutators of two matrices.
\par
These calculations are {\it elementary\/},
in the strict technical sense of \cite{vavplot}. But being elementary,
they are by no means easy. Sometimes these calculations are even called
the {\bf yoga of conjugation}, and the {\bf yoga of commutators},
to stress the overwhelming feeling of technical strain and exertion.
\par
A specific motivation for the present work was the desire to create
tools to prove {\it relative\/} versions of structure results for Chevalley
groups. Here we list three such immediate applications, in which we were
particularly interested.
\par
\smallskip
$\bullet$ Description of subnormal subgroups and subgroups normalised
by the relative elementary subgroup. In full generality such description
is only available for classical groups \cite{ZZ,ZZ1,ZZ2,youhong}, but,
apart from the case of $\GL(n,R)$
\cite{wilson72,bak82,vaser86,LL,vavilov90,vaser90}, sharp bounds are not
obtained even in this case.
\par\smallskip
$\bullet$ Results on description of intermediate subgroups, such as,
for example, overgroups of regularly embedded semi-simple subgroups,
overgroups of exceptional Che\-valley groups in an appropriate $\GL(n,R)$,
etc., see, for example, \cite{luzgarev2004,luzgarev2008,SVY} and
\cite{LV,NV95,VSsamara} for a survey and further references.
\par\smallskip
$\bullet$ Generalisation of the mixed commutator formula
$$ \big[E(n,R,\ma),\GL(n,R,\mb)\big]=\big[E(n,R,\ma),E(n,R,\mb)\big], $$
\par\noindent
to exceptional Chevalley groups.
\par
The first two problems are discussed in somewhat more detail in the last
section, complete proofs are relegated to subsequent papers by the authors.
Here we discuss only the third one, relative standard commutator formulae,
another major objective of the present paper, apart from developing the
localisation machinery itself.
\par
The above formula was proved in the setting of general linear groups by
Alexei Stepanov and the second author~\cite{VS8}. This formula is a common
generalisation of both absolute standard commutator formulae.
At the stable level, absolute commutator formulae were first established
in the foundational work of Hyman Bass~\cite{Bass2}.
In another decade, Andrei Suslin, Leonid Vaserstein, Zenon Borewicz,
and the second author \cite{Sus,V1,BV85,SV} discovered that for
commutative rings similar formulae hold for all $n\geq 3$.
For two relative subgroups such formulae were proven only at the
stable level, by Alec Mason \cite{Mason74} -- \cite{MAS3}.
\par
However, the proof in \cite{VS8} relied on a {\it very\/} strong
and precise form of decomposition of unipotents \cite{SV}, and was
not likely to easily generalise to groups of other types. Stepanov
and the second-named author raised the following problems.
\par\smallskip
$\bullet$ Establish the relative standard commutator formula via
localisation method \cite[Problem~2]{VS8}.
\par\smallskip
$\bullet$ Generalise the relative standard commutator formula to
Bak's unitary groups and to Chevalley groups \cite[Problem~1]{VS8}.
\par\smallskip
In the paper \cite{RZ} the first and the third authors developed relative
versions of conjugation calculus and commutator calculus in the
general linear group $\GL(n,R)$, thus solving \cite[Problem~2]{VS8}. In
\cite{RNZ} we developed a similar relative conjugation calculus in Bak's
unitary groups, thus accounting for all {\it even\/} classical groups.
\par
In the present paper, which is a direct sequel of \cite{RZ,RNZ}, we
in a similar way evolve relative conjugation calculus and commutator
calculus in arbitrary Chevalley groups.
Actually, the present
paper does not depend on the calculations from \cite{RN1,SV10}. Instead,
here we develop {\it relative\/} versions of the yoga of conjugation,
and the yoga of commutators {\it from scratch\/}, in a more general setting.
The reason is that in the relative setting it is not enough to prove
the continuity of conjugation by $g$. What we now need, is its
{\it equi-continuity\/}  on all congruence subgroups $G(\Phi,R,I)$.
In other words, we need explicit bounds for the modulus of continuity,
{\it uniform\/} in the ideal $I$.
The resulting versions of conjugation calculus and commutator calculus
are both {\it substantially\/} more powerful, and {\it easier\/} than
the ones available in the literature.
\par
The overall scheme is always the same as devised by the first and the
second authors in \cite{RN1} (which, in turn, was a further
elaboration of \cite{B4,RH,RH2}), and as later implemented by Alexei Stepanov
and the second author \cite{SV10} in a slightly more precise version,
with length bounds. However, we propose several major technical
innovations, and simplifications. Most importantly, following \cite{RZ}
and \cite{RNZ} we construct another base of $s$-adic neighbourhoods
ot 1, consisting of {\it partially\/} relativised elementary
groups, and prove all results not at the absolute, but at the
relative level.
\par
As an immediate application of our methods we prove the following
result which, together with \cite{RNZ}, solves \cite[Problem 1]{VS8}
and \cite[Problem 4]{BRN}. Specifically, for Chevalley groups the
same question was reiterated as \cite[Problem 6]{RNZ}. Definitions
of the elementary subgroup $E(\Phi,R,\ma)$ and the full congruence
subgroup $C(\Phi,R,\ma)$ of level $\ma\unlhd R$ are recalled in
\S\S~3,~4.
\begin{The}\label{main}
Let\/ $\Phi$ be a reduced irreducible root system,\/ $\rk(\Phi)\ge 2$.
Further, let\/ $R$ be a commutative ring, and\/ $\ma,\mb\unlhd R$ be
two ideals of\/ $R$. In the cases\/ $\Phi=\C_2,\G_2$ assume
that\/ $R$ does not have residue fields\/ ${\Bbb F}_{\!2}$ of\/ $2$
elements and in the case\/ $\Phi=\C_l$, $l\ge 2$, assume additionally
that any\/ $c\in R$ is contained in the ideal\/ $c^2R+2cR$. Then
$$ [E(\Phi,R,\ma),C(\Phi,R,\mb)]=[E(\Phi,R,\ma),E(\Phi,R,\mb)]. $$
\end{The}
Actally, as we shall see in \S~10, this commutator formula is
equivalent to a slightly weaker formula
$$ [E(\Phi,R,\ma),G(\Phi,R,\mb)]=[E(\Phi,R,\ma),E(\Phi,R,\mb)]. $$
\noindent
Before, for exceptional groups this theorem was known in the
two special cases\footnote{Actually, after submitting the present paper
we learned a very important paper by You Hong \cite{you92}, which
contains essentially the same result, with a proof very close in
spirit to our {\it second\/} proof here. A slight technical difference
is that \cite{you92} relies on straightforward commutator identities
for individual elements, whereas we invoke the three subgroup lemma,
which makes the argument slightly shorter and more transparent.
Also, as too many other publications, \cite{you92} contains a minor
inaccuracy --- one of the {\it hazar\/} = {\it thousand\/} traps, of
which Omar speaks! --- in that an extra condition is imposed only in
the case $\C_2$, whereas it is requisite for all $\C_l$, $l\ge 2$.
{\it Fortunately\/}, we were not aware of \cite{you92}, when writing
\cite{RZ,RNZ} and the present paper. Otherwise, we would had been much
less eager to develop a {\it localisation\/} approach towards the
proof of Theorem 1. We are convinced that the main contribution of the
present paper are the relative versions of conjugation calculus,
commutator calculus, and patching, developed in \S\S~7--9. They already
have several further important applications, which go well beyond
Theorem 1 or the main results of \cite{RN1,BRN} and \cite{SV10}.},
where $R=\ma$ or where $R=\mb$, see \cite{taddei,vaser2}.
\par
With the above precise condition Theorem 1 is proven in
\S~10. Our localisation proof in \S~9 requires a somewhat stronger
condition $2\in R^*$, in the cases $\Phi=\C_2,\G_2$. Strictly speaking,
this stronger condition is not necessary, we only use it to simplify the
proof of the induction base of the relative commutator calculus in \S~8.
\par
This small compromise allows us to spare some 5--6 pages of
calculations, and to eventually develop a more technical and powerful
version of relative localisation with {\it two\/} denominators.
As a matter of fact, the main result of the present paper is not
the above Theorem 1 itself, but rather Theorem 2 established in \S~9.
Theorem 2 looks too technical to stand well on its own, but actually
it is {\it terribly\/} much stronger and more general than Theorem 1.
It is devised to be used in our subsequent publications to derive
{\it multiple\/} commutator formulae, which are simultaneous
generalisations of Theorem 1 and nilpotency of $K_1$.
\par
However, not to further complicate things, we decided to relegate the
detailed analysis of the rank 2 cases to a subsequent publication,
especially that it should be carried in a more general setting,
much more technically demanding. In the meantime, let us explain,
why the rank 2 case, namely the types $\C_2$ and $\G_2$, require
some serious extra care. This is due to the following circumstances.
\par\smallskip
$\bullet$ In these cases, the elementary group $E(\Phi,R)$ is not
perfect when $R$ has residue field\/ ${\Bbb F}_{\!2}$, which accounts
for the first assumption in Theorem 1.
\par\smallskip
$\bullet$ There is {\it substantially\/} less freedom in the Chevalley
commutator formula, especially for groups of type $\C_2$, which accounts
for the additional assumption in this case.
\par\smallskip
$\bullet$ There is somewhat less freedom also in the choice of
semi-simple factors.
\par\smallskip
$\bullet$ Most importantly, in these cases it is natural to define
relative subgroups not in terms of ideals, but in terms of
form ideals, or even more general structures, such as radices
\cite{costakeller1,costakeller2}.
\par\smallskip
As in \cite{RNZ}, in the present paper we concentrate on actual
calculations. The history of localisation methods, the philosophy behind
them, and their possible applications are extensively discussed in our
mini-survey with Alexei Stepanov \cite{yoga}. There, we also describe
another remarkable recent advance, {\it universal\/} localisation
developed by Stepanov \cite{step}. For {\it algebraic\/} groups,
universal localisation allows --- among other things --- to remove
dependence on the dimension of the ground ring $R$ in the results of
\cite{SV10}. Unfortunately, generalised unitary groups are not always
algebraic, so that our width bounds for commutators in unitary
groups \cite{HSVZ} still depend on $\dim(\Max(R))$.
\par
The paper is organised as follows. In \S\S~2--4 we recall basic
notation, and some background facts, used in the sequel.
In \S~5 we discuss injectivity of localisation homomorphism
and in \S~6 we calculate levels of mixed commutator subgroups.
The next two sections constitute the technical core of the paper.
Namely, in \S~7, and in \S~8 we develop relative conjugation
calculus, and relative commutator calculus in Chevalley groups,
respectively. After that we are in a position to give a localisation
proof of Theorem~1 --- and in fact of a much stronger Theorem 2 ---
in \S~9. On the other hand, using level calculations in \S~10 we
give another proof of Theorem~1, deducing it from the {\it absolute\/}
standard commutator formula. There we also obtain slightly more
precise results in some special situations, such as Theorem 3,
which completely calculates the relative commutator subgroup in the
important case, where $\ma$ and $\mb$ are comaximal, $\ma+\mb=R$.
Finally, in \S~11 we state and briefly review some further related
problems.


\section{Chevalley groups}

As above, let $\Phi$ be a reduced irreducible root system of
rank $l=\rk(\Phi)$, and $P$, $Q(\Phi)\le P\le P(\Phi)$ be a
lattice between the root lattice $Q(\Phi)$ and the weight
lattice $P(\Phi)$. Usually, we fix an order on $\Phi$
and denote by $\Pi=\{\a_1,\ldots,\a_l\}$ $\Phi^+$, $\Phi^-$
the corresponding sets of fundamental, positive, and negative
roots, respectively. Recall, that
$Q(\Phi)=\Int\a_1\oplus\ldots\oplus\Int\a_l$
and $P(\Phi)=\Int\varpi_1\oplus\ldots\oplus\Int\varpi_l$,
where $\varpi_!,\ldots,\varpi_l$ are the corresponding
fundamental weights. Finally, $W=W(\Phi)$ denotes the Weyl group
of $\Phi$.
\par
Further, let $R$ be a commutative ring. We denote by
$G=G_P(\Phi,R)$ the Chevalley group of type $(\Phi,P)$ over $R$,
by $T=T_P(\Phi,R)$ a split maximal torus of $G$ and by
$E=E_P(\Phi,R)$ the corresponding (absolute) elementary
subgroup. Usually $P$ does not play role in our calculations
and we suppress it in the notation.
\par
The elementary group $E(\Phi,R)$ is generated by all root
unipotents $x_{\alpha}(a)$, $\alpha\in\Phi$, $a\in R$,
elementary with respect to $T$. The fact that $E$ is normal
in $G$ means exactly that $E$ does not depend on the choice of $T$.
\par
Let $G$ be a group. For any $x,y\in G$,  $^xy=xyx^{-1} $  denotes the
left $x$-conjugate of $y$. Let  $[x,y]=xyx^{-1}y^{-1}$ denote the
commutator of $x$ and $y$. We will make frequent use of the
following formulae,
\par\smallskip
(C1) $[x,yz]=[x,y]\cdot{}^y[x,z]$,
\par\smallskip
(C2) $[xy,z]={}^x[y,z]\cdot[x,z]$,
\par\smallskip
(C3) Hall---Witt identity
$$ {}^x\bigl[[x^{-1},y],z\bigr]={}^x\bigl[[y,x^{-1}]^{-1},z\bigr]=
{}^y\bigl[x,[y^{-1},z]\bigr]\cdot{}^z\bigl[y,[z^{-1},x]\bigr], $$
\par
(C4) $[x,{}^yz]={}^y[{}^{y^{-1}}x,z]$,
\par\smallskip
(C5) $[{}^yx,z]={}^{y}[x,{}^{y^{-1}}z]$.
\par\smallskip
Most of the calculations in the present paper
are based on the Steinberg relations
\par\smallskip
(R1) Additivity of $x_{\a}$,
$$ x_{\a}(a+b)=x_{\a}(a)x_{\a}(b). $$
\par
(R2) Chevalley commutator formula
$$ [x_{\al}(a),x_{\beta}(b)]=\prod_{i\al+j\beta\in \Phi}
x_{i\al+j\beta}(N_{\al\beta ij}a^ib^j), $$
\noindent
where $\al\not=-\beta$ and $N_{\al\beta ij}$ are the structure constants
which do not depend on $a$ and $b$. Notice, though, that for $\Phi=\G_2$
they may depend on the order of the roots in the product on the right
hand side. The following
observation was made by Chevalley himself: let $\al-p\beta,\ldots,\al
-\beta,\al,\al+\beta,\ldots,\al+q\beta$ be the $\al$-series of roots
through $\beta$, then $N_{\al\beta 11}=\pm(p+1)$ and
$N_{\al\beta12}=\pm(p+1)(p+2)/2$.
\par\smallskip
Let $i_{\Phi}$ be the largest integer which may appear as $i$ in a
root $i\al+j\beta\in\Phi$ for all $\al,\beta\in\Phi$. Obviously
$i_{\Phi}=1,2$ or 3, depending on whether $\Phi$ is simply laced,
doubly laced or triply laced.  The following result makes the proof
for $\Phi\neq C_l$ slightly easier than for the symplectic
case. Recall that $\A_1=\C_1$ and $\B_2=\C_2$ so that root systems of
types $\A_1$ and $\B_2$ {\it are\/} symplectic. All roots of $\A_1$ are
long.
\par
Our calculations in \S~7 and \S~8 rely on the following result, which
is Lemma~2.12 in \cite{RN1}.
\begin{Lem}\label{lem1}
Let\/ $\beta\in\Phi$ and either\/ $\Phi\neq\C_l$ or\/ $\beta$
is short. Then there exist two roots\/ $\gamma,\delta\in\Phi$ such
that\/ $\beta=\gamma+\delta$ and\/ $N_{\gamma\delta11}=1$.
\par
If\/ $\Phi=\C_l$, $l\ge 2$, and\/ $\beta$ is long, then there exist two
roots\/ $\gamma,\delta\in\Phi$ such that either\/ $\beta=\gamma+2\delta$
and\/ $N_{\gamma\delta12}=1$, or\/ $\beta=2\gamma+\delta$
and\/ $N_{\gamma\delta21}=1$.
\end{Lem}
In the sequel we also use semi-simple root elements. Namely,
for $\alpha\in\Phi$ and $\e\in R^*$ we set
$$ w_\a(t)=x_\a(t)x_{-\a}(-t^{-1})x_\a(t),\qquad
h_\a(t)=w_\a(t)w_\a(1)^{-1}. $$
\noindent
Let $H(\Phi,R)$ be the subgroup of $T(\Phi,R)$, generated by
all $h_{\a}(\e)$, $\a\in\Phi$, $\e\in R^*$.
\par
Clearly, $H(\Phi,R)\le E(\Phi,R)$, and in fact,
$H(\Phi,R)=T(\Phi,R)\cap E(\Phi,R)$. In particular, for simply
connected group one has
$$ H_{\sic}(\Phi,R)=T_{\sic}(\Phi,R)=\Hom(P(\Phi),R^*). $$
\noindent
For non simply connected groups, specifically, for the adjoint ones,
$T(\Phi,R)$ is usually somewhat larger, than $H(\Phi,R)$. For
the proof of our main theorem we have to understand, what the
generators of $T(\Phi,R)$ look like in this case, see
\cite{weight-elem} for explicit constructions and many further
references.
\par
Let $\omega\in P(\Phi^{\vee})$, by definition
$(\a,\omega)\in {\Bbb Z}$ for all $\a\in\Phi$. The adjoint torus
contains weight elements $h_{\omega}(\e)$, which commute with all
elements from $T$ and satisfy the following commutator relation:
$$ h_{\omega}(\e)x_{\a}(\xi)h_{\omega}(\e)^{-1}=
x_{\a}(\e^{(\a,\o)}\xi), $$
\noindent
for all $\a\in\Phi$ and all $\xi\in R$. For $\Phi=\E_8,\F_4$ and
$\G_2$, one has $P(\Phi)=Q(\Phi)$, in particular, in these cases
$T_{\ad}(\Phi,R)=H_{\ad}(\Phi,R)$. For other cases $T_{\ad}(\Phi,R)$
is generated by $H_{\ad}(\Phi,R)$ and some weight elements.
\par
In Section~\ref{hgfdsa} we can refer to either one of the following lemmas.
The first one follows from \cite{weight-elem}, Proposition 1, while
the second one is well-known and obvious.
\begin{Lem}\label{lem-weights}
The torus $T_{\ad}(\Phi,R)$ is generated by $H_{\ad}(\Phi,R)$
and weight elements $h_{\omega}(\e)$, where $\e\in R^*$, and
$\omega$ are the following weights
\par\smallskip
$\bullet$ $\omega=\varpi_1$, for\/ $\Phi=\A_,\B_l$ and $\E_6$,
\par\smallskip
$\bullet$ $\omega=\varpi_l$, for\/ $\Phi=\C_l$,
\par\smallskip
$\bullet$ $\omega=\varpi_1,\varpi_{l}$, for\/ $\Phi=\D_l$,
\par\smallskip
$\bullet$ $\omega=\varpi_7$, for\/ $\Phi=\E_7$.
\end{Lem}\label{lem-conjug}
\begin{Lem}\label{lem-diagonal}
Assume that either $\Phi\neq\C_l$, or $a\in\Phi$ is short. Then
for any $\e\in R^*$ there exists an $h\in H(\Phi,R)$ such that
$hx_{\a}(\xi)h^{-1}=x_{\a}(\e\xi)$, for all $\xi\in R$.
\par
In the exceptional case, where $\Phi=\C_l$ and $a\in\Phi$ is long,
$hx_{\a}(\xi)h^{-1}=x_{\a}(\e^2\xi)$, for all $h\in H(\Phi,R)$.
On the other hand, if $\a\in\Phi^+$ is a positive long root,
$$ h_{\varpi_l}(\e)x_{\a}(\xi)h_{\varpi_l}(\e)^{-1}=x_{\a}(\e\xi). $$
\end{Lem}
Clearly, in the last case for a negative long root one has
$h_{\varpi_l}(\e)x_{\a}(\xi)h_{\varpi_l}(\e)^{-1}=x_{\a}(\e^{-1}\xi)$.
In the vector representation of the {\it extended\/} simply connected
Chevalley group $\bar G(\C_l,R)=\GSp(2l,R)$ this weight element
has the form
$$ h_{\varpi_l}(\e)=\diag(\e,\ldots,\e,1,\ldots,1). $$
\par
It follows that --- with the only possible exception when $\Phi=\C_l$
and $\a$ is long --- for any $\a\in\Phi$ and any $h\in T(\Phi,R)$
there exists a $g\in H(\Phi,R)$ such that
$gx_{\a}(\xi)g^{-1}=hx_{\a}(\xi)h^{-1}$. In particular, $g^{-1}h$
commutes with $x_{\a}(\xi)$. However, in the exceptional case,
where $\Phi=\C_l$ and $\a$ is long, no such $g$ exists in general.
One can only ensure the existence of such a $g\in H(\Phi,R)$ that
$g^{-1}h=h_{\varpi_l}(\e)$ for some $\e\in R^*$.


\section{Relative elementary subgroups}

In this section we recall the definitions of relative subgroups,
and some basic facts used in the sequel. The usual one-parameter
relative subgroups are well known. However, for multiply laced
systems one should consider two-parameter relative subgroups,
with one parameter corresponding to short roots, and another one to
long roots. Such two-parameter relative subgroups were introduced
and studied by Eiichi Abe \cite{abe0}--\cite{abe1} and Michael Stein
\cite{stein2}.
\par
Let $\ma$ be an additive subgroup of $R$. Then $E(\Phi,\ma)$ denotes
the subgroup of $E$ generated by all elementary root unipotents
$x_{\al}(t)$ where $\al\in\Phi$ and $t\in\ma$.  Further, let $L$
denote a nonnegative integer and let $E^L(\Phi,\ma)$ denote the {\it
subset\/} of $E(\Phi,\ma)$ consisting of all products of $L$ or fewer
elementary root unipotents $x_{\al}(t)$, where $\al\in\Phi$ and
$t\in\ma$. In particular, $E^1(\Phi,\ma)$ is the set of all $x_{\al}(t)$,
$\al\in\Phi$, $t\in\ma$.
\par
When $\ma\trianglelefteq R$ is an ideal of $R$, the elementary group
$E(\Phi,\ma)$ of level $\ma$ should be distinguished from the
{\it relative\/} elementary subgroup $E(\Phi,R,\ma)$ of level $\ma$.
By definition $E(\Phi,R,\ma)$ is the normal closure of $E(\Phi,\ma)$
in the absolute elementary subgroup $E(\Phi,R)$. In general
$E(\Phi,R,\ma)$ is {\it not\/} generated by elementary
transvections of level $\ma$. Below we describe its generators for
$\rk(\Phi)\ge 2$. The following result can be found in
\cite{stein2,tits}.

\begin{Lem}\label{lemtits}
In the case\/ $\Phi\neq\C_l$ one has\/ $E(\Phi,\ma)\ge
E(\Phi,R,\ma^2)$. In the exceptional case\/ $\Phi=\C_l$ one
has\/ $E(\Phi,\ma)\ge E(\Phi,R,(2R+\ma)\ma^2)$.
\end{Lem}

\par
Let $\ma$ be an ideal of $R$. Denote by $\ma_2$ the ideal, generated by
$2\xi$ and $\xi^2$ for all $\xi\in\ma$. The first component $\ma$ of an
{\it admissible pair} $(\ma,\mb)$ is an ideal of $R$, parametrising short
roots. When $\Phi\neq\C_l$ the second component $\mb$,
$\ma_2\le\mb\le\ma$, is also an ideal, parametrising long roots.
In the exceptional
case $\Phi=\C_l$ the second component $\mb$ is an additive subgroup
stable under multiplication by $\xi^2$, $\xi\in R$ (in other words,
it is a relative form parameter in the sense of Bak \cite{BV3,HO,RN}).
A similar notion can be introduced for the type $\G_2$ as well, but
in this case one should replace 2 by 3 {\it everywhere\/} in the above
definition.
\par
Now the {\it relative\/} elementary subgroup, corresponding to an
admissible pair $(\ma,\mb)$, is defined as follows:
$$ E(\Phi,R,\ma,\mb)={\left\langle
x_{\alpha}(\xi),\alpha\in\Phi_s,\xi\in\ma;\
x_{\beta}(\zeta),\beta\in\Phi_l,\zeta\in\mb
\right\rangle}^{E(\Phi,R)}. $$
\noindent
where $\Phi_s$ and $\Phi_l$ are the sets of long and short
roots in $\Phi$, respectively. The following results can be found in
\cite{stein2,abe89,abe91}.

\begin{Lem}\label{lemma1}
Let\/ $\rk(\Phi)\ge 2$. When\/ $\Phi=\operatorname{B}_2$
or\/ $\Phi=\operatorname{G}_2$ assume moreover that\/ $R$ has
no residue fields\/ ${\Bbb F}_{\!2}$ of $2$ elements. Then the
elementary subgroup\/ $E(\Phi,R,\ma,\mb)$ is\/ $E(\Phi,R)$-perfect,
in other words,
$$ \big[E(\Phi,R),E(\Phi,R,\ma,\mb)\big]=E(\Phi,R,\ma,\mb). $$
\noindent
In particular,\/ $E(\Phi,R)$ is perfect.
\end{Lem}

\begin{Lem}\label{lemma2}
As a subgroup\/ $E(\Phi,R,\ma,\mb)$ is generated by the elements
$$ z_{\al}(\xi,\zeta)=x_{-\al}(\zeta)x_{\al}(\xi)x_{-\al}(-\zeta), $$
\noindent
where\/ $\xi\in\ma$ for $\al\in\Phi_s$ and\/ $\xi\in\mb$
for\/ $\al\in\Phi_l$, while\/ $\zeta\in R$.
\end{Lem}
Actually, in the sequel we mostly use these results in the
special case, where $\ma=\mb$.


\section{Congruence subgroups}\label{level2}

Usually, one defines congruence subgroups as follows. An ideal
$\ma\trianglelefteq R$ determines the reduction homomorphism
$\rho_{\ma}:R\map R/\ma$. Since $G(\Phi,\underline{\ \ })$ is a functor from
rings to groups, this homomorphism induces reduction homomorphism
$\rho_{\ma}:G(\Phi,R)\map G(\Phi,R/\ma)$.
\par\smallskip
$\bullet$ The kernel of the reduction homomorphism $\rho_{\ma}$
modulo $\ma$ is called the principal congruence subgroup of
level $\ma$ and is denoted by $G(\Phi,R,\ma)$.
\par\smallskip
$\bullet$ The full pre-image of the centre of $G(\Phi,R/\ma)$ with
respect to the reduction homomorphism $\rho_{\ma}$ modulo $\ma$
is called the full congruence subgroup of level $\ma$, and
is denoted by $C(\Phi,R,\ma)$.
\par\smallskip
A more general notion of congruence subgroup was introduced in
\cite{HPV}.
Namely, consider a linear action of $G$ on a right $R$-module $V$
and let $U\le V$ be a $G$-submodule. Then we can define a set
$$ G(V,U)=\big \{g\in G\mid \forall v\in V,\ gv-v\in U\big \}. $$
\noindent
This set is in fact a {\it normal\/} subgroup of $G$.
\par
An application of this construction to a Chevalley group
$G=G(\Phi,R)$ and its rational module $V$ allows us to recover
the usual subgroups. For any module $V$ and any ideal
$\ma\trianglelefteq R$ the product $U=V\ma$ is a $G$-submodule.
The following result is \cite[Lemma~6]{HPV}.
\begin{Lem}\label{lemma6}
When\/ $V$ is a faithful rational module,\/ $G(V,V\ma)=G(\Phi,R,\ma)$
is the usual principal congruence subgroup of level\/ $\ma$.
\end{Lem}
In matrix language, this lemma means that the principal
congruence subgroup of level $\ma$ can be defined as
$$ G(\Phi,R,\ma)=G(\Phi,R)\cap\GL(n,R,\ma), $$
\noindent
for any {\it faithful\/} rational representation
$G(\Phi,R)\le\GL(n,R)$.
\par
Clearly, for {\it any\/} rational representation
$\phi:G(\Phi,R)\map\GL(n,R)$, one has the inclusions
$$ \phi^{-1}\big(G(\Phi,R)\cap\GL(n,R,\ma)\big)\le
C(\Phi,R,\ma)\le \phi^{-1}\big(G(\Phi,R)\cap C(n,R,\ma)\big), $$
\noindent
for the full congruence subgroup. In the general case
there is no reason, why either of these inclusions should be
an equality. However, there is one important special case,
where the left inclusion becomes an equality \cite[Lemma~7]{HPV}.
\begin{Lem}\label{lemma7}
When\/ $V=L$ is the Lie algebra of\/ $G(\Phi,R)$, considered as the
adjoint module, then\/ $G(L,L\ma)=C(\Phi,R,\ma)$ is the usual full
congruence subgroup of level\/ $\ma$.
\end{Lem}
The following result, Theorem 2 of \cite{HPV}, asserts that three
possible definitions of the full congruence subgroup coincide.
\begin{Lem}\label{thm2}
Let\/ $\Phi$ be a reduced irreducible root system of rank\/ $\ge 2$,
$R$ be a commutative ring,\/ $(\ma,\mb)$ an admissible pair. Then the
following four subgroups coincide:
$$ \aligned C(\Phi,R,\ma,\mb)
&=\big \{g\in G(\Phi,R)\mid [g,E(\Phi,R)]\le E(\Phi,R,\ma,\mb)\big \}\\
&=\big\{g\in G(\Phi,R)\mid [g,E(\Phi,R)]\le C(\Phi,R,\ma,\mb)\big\}\\
&=\big\{g\in G(\Phi,R)\mid [g,G(\Phi,R)]\le C(\Phi,R,\ma,\mb)\big\}.\\
\endaligned $$
\end{Lem}
In fact, in \cite{HPV} we established standard commutator formulae
for the case, where one argument is an {\it absolute\/} subgroup,
whereas the second argument is a relative subgroup with {\it two\/}
parameters. In particular, the following result is Theorem 1
of \cite{HPV}. Of course, in all cases, except Chevalley groups
of type $\F_4$, it was known before, \cite{BV3,petrov2,costakeller2}.
\begin{Lem}\label{thm1}
Let\/ $\Phi$ be a reduced irreducible root system of rank\/ $\ge 2$,
$R$ be a commutative ring,\/ $(\ma,\mb)$ an admissible pair. In the
case, where\/ $\Phi=\C_2$ or\/ $\Phi=\G_2$ assume moreover
that\/ $R$ has no residue fields\/ ${\Bbb F}_{\!2}$ of\/ $2$
elements. Then the following standard commutator formulae holds
$$ \big [G(\Phi,R),E(\Phi,R,\ma,\mb)\big ]=\big [E(\Phi,R),C(\Phi,R,\ma,\mb)\big]=
E(\Phi,R,\ma,\mb). $$
\end{Lem}
We will use the following form of Gau{\ss} decomposition, stated
by Eiichi Abe \cite{abe0,abe1}. Namely, let $\ma\unlhd R$ be an
ideal of $R$. We denote by $T(\Phi,R,\ma)$ the subgroup of the
split maximal torus $T(\Phi,R)$, consisting of all elements
congruent to $e$ modulo $\ma$,
$$ T(\Phi,R,\ma)=T(\Phi,R)\cap G(\Phi,R,\ma). $$
\noindent
As usual, we set
$$ U(\Phi,\ma)=\big\langle x_{\a}(a),\ \a\in\Phi^+,\ a\in\ma\big\rangle,
\qquad
U^-(\Phi,\ma)=\big\langle x_{\a}(a),\ \a\in\Phi^-,\ a\in\ma\big\rangle. $$
\noindent
Obviously, $U(\Phi,\ma),U^-(\Phi,\ma)\le E(\Phi,\ma)$.
\begin{Lem}
\label{abe}
Let\/ $\ma$ be an ideal of\/ $R$ contained in the
Jacobson radical\/ $\Rad(R)$. Then
$$ G(\Phi,R,\ma)=U(\Phi,\ma)T(\Phi,R,\ma)U^-(\Phi,\ma). $$
\end{Lem}
We will mostly use this lemma in the following form, see \cite{RN1},
Lemma 2.10.
\begin{Lem}
\label{abe-bis}
If\/ $\ma$ is an ideal of local ring\/ $R$ then
$$ G(\Phi,R,\ma)=E(\Phi,\ma)T(\Phi,R,\ma). $$
\end{Lem}


\section{Injectivity of localisation homomorphism}

Let us fix some notation. Let $R$ be a commutative ring with
1, $S$ be a multiplica\-tive system in $R$ and $S^{-1}R$ be the
corresponding localisation. We will mostly use localisation
with respect to the two following types of multiplicative
systems.
\par\smallskip
$\bullet$ If $s\in R$ and the multiplicative system $S$ coincides
with $\langle s\rangle=\{1,s,s^2,\ldots\}$ we usually write
$\langle s\rangle^{-1}R=R_s$.
\par\smallskip
$\bullet$ If $\mm\in\Max(R)$ is a maximal ideal in $R$, and
$S=R\setminus\mm$, we usually write $(R\setminus\mm)^{-1}R=R_\mm$.
\par\smallskip
We denote by $F_S:R\map S^{-1}R$ the canonical ring homomorphism
called the localisation homomorphism. For the two special cases
mentioned above, we write $F_s:R\map R_s$ and $F_\ma:R\map R_\ma$,
respectively.
\par
When we write an element as a fraction, like $a/s$ or
$\displaystyle{\frac{a}{s}}$, we {\it always\/} think of it
as an element of some localisation $S^{-1}R$, where $s\in S$.
If $s$ were actually invertible in $R$, we would have written
$as^{-1}$ instead.

\begin{fragment}\label{dirlim}
The property of these functors which will be crucial for what follows
is that they {\it commute with direct limits\/}. In other words, if
$R=\varinjlim R_i$, where $\{R_i\}_{i\in I}$ is an inductive
system of rings, then
$X(\Phi,\varinjlim R_i)=\varinjlim X(\Phi,R_i)$. We
will use this property in the following two situations.
\par\smallskip
$\bullet$ First, let $R_i$ be the inductive system of all finitely
generated subrings of $R$ with respect to inclusion. Then
$X=\varinjlim X(\Phi,R_i)$, which reduces most of the proofs
to the case of Noetherian rings.
\par\smallskip
$\bullet$ Second, let $S$ be a multiplicative system in $R$ and $R_s$, $s\in S$,
the inductive system with respect to the localisation homomorphisms:
$F_{t}:R_s\map R_{st}$. Then $X(\Phi,S^{-1}R)=\varinjlim X(\Phi,R_s)$,
which allows to reduce localisation in any multiplicative system to
principal localisations.
\end{fragment}

Our proofs rely on the injectivity of localisation homomorphism
$F_s$. On the group $G(\Phi,R)$ itself it is seldom injective,
but its restrictions to appropriate congruence subgroups often are,
see the discussion in \cite{yoga}. Below we cite two important
typical cases, Noetherian rings \cite{B4} and semi-simple rings
\cite{SV10}.
\begin{Lem}
\label{noetherian}
Suppose\/ $R$ is Noetherian and\/ $s\in R$. Then
there exists a natural number $k$ such that the homomorphism\/
$F_s:G(\Phi,R,s^kR)\map G(\Phi,R_s)$ is injective.
\end{Lem}
\begin{proof}
The homomorphism $F_s:G(\Phi,R,s^kR)\map G(\Phi,R_s)$ is
in\-jec\-ti\-ve when\-ever $F_s:s^kR\map R_s$ is injective.
Let $\ma_i=\Ann_R(s^i)$ be the annihilator of $s^i$ in $R$.
Since $R$ is Noetherian, there exists $k$ such that
$\ma_k=\ma_{k+1}=\ldots$. If $s^ka$ vanishes in $R_s$, then
$s^is^ka=0$ for some $i$. But since $\ma_{k+i}=\ma_k$, already
$s^ka=0$ and thus $s^R$ injects in $R_s$.
\end{proof}
\begin{Lem}
\label{semisimple}
If\/ $\Rad(R)=0$, then\/ $F_s:G(\Phi,R,sR)\map G(\Phi,R_s)$
is injective for all\/ $s\in R$, $s\neq 0$.
\end{Lem}
\begin{proof}
It suffices to prove that $F_s:sR\map R_s$ is injective.
Suppose that $s\xi\in sR$ goes to $0$ in $R_s$. Then there
exists an $m\in\Nat$ such that $s^ms\xi=0$. It follows that
$(s\xi)^{m+1}=0$ and since $R$ is semi-simple, $s\xi=0$.
\end{proof}
In \cite{RN1} we used reduction to Noetherian rings, whereas in
\cite{SV10} reduction to semi-simple rings was used instead.
\goodbreak


\section{Levels of mixed commutators}\label{level}

In this section we establish some obvious facts, concerning
the lower and the upper levels of mixed commutators
$$ \big[E(\Phi,R,\ma),E(\Phi,R,\mb)\big]\le \big[G(\Phi,R,\ma),G(\Phi,R,\mb)\big]. $$
\par
Unlike most other results of the present paper, the next lemma
also holds for $\rk(\Phi)=1$.
\begin{Lem}
Then for any two ideals\/ $\ma$ and\/ $\mb$ of the ring\/ $R$ one has
$$ E(\Phi,R,\ma,\mc)E(\Phi,R,\mb,\md)=E(\Phi,R,\ma+\mb,\mc+\md). $$
\end{Lem}
\begin{proof}
Additivity of the elementary root unipotents
$x_{\a}(a+b)=x_{\a}(a)x_{\a}(b)$, where $\a\in\Phi$,
while $a\in\ma$, $b\in\mb$ for $\a\in\Phi_s$, and
$a\in\mc$, $b\in\md$ for $\a\in\Phi_l$, implies that the left
hand side contains {\it generators\/} of the right hand side.
The product of two normal subgroups is normal in $E(\Phi,R)$.
\end{proof}
As a preparation to the calculation of lower level, we generalise
Lemma~\ref{lemtits}. It is a toy version of the main results
of the present paper, whose proof heavily depends on Lemma~\ref{lemma2}.
There one considers
$$ z_{\a}(ab,\zeta)={}^{x_{-\a}(\zeta)}x_{\a}(ab). $$
\noindent
The idea is to express $x_{\a}(ab)$ as the commutator of two root
elements $[x_{\b}(a),x_{\gamma}(b)]$, where $\b+\gamma=\a$, plus,
possibly, some tail. Now, neither the roots $\b,\gamma$, nor the
roots appearing in the tail, are opposite to $-\a$, and thus we
can distribute conjugation by $x_{-\a}(\zeta)$ and apply the
Chevalley commutator formula to each occuring factor. The first
explicit appearance of this idea, which we were able to trace in
the literature, was in Bass---Milnor---Serre foundational work
\cite{bassmilnorserre}.
\par
In the following lemma, we should distinguish the ideal $\ma^2$,
generated by the products $ab$, where $a,b\in\ma$, from the ideal
$\ma\supb{2}$, generated by $a^2$, where $a\in\ma$. Clearly,
when $2\in R^*$ these ideals coincide, but this case is trivial
anyway.
\begin{Lem}\label{nllnnn}
Let\/ $\rk(\Phi)\ge 2$ and further let\/ $\ma$ and\/ $\mb$ be two
ideals of\/ $R$. Assume that either $\Phi\neq\C_l$, or $2\in R^*$.
Then one has
$$ E(\Phi,R,\ma\mb)\le E(\Phi,\ma+\mb). $$
\noindent
In the exceptional case, where $\Phi=\C_l$ and $2\notin R^*$ one has
$$ E(\Phi,R,\ma\mb,\ma\mb\supb{2}+2\ma\mb+\ma\supb{2}\mb)\le 
E(\Phi,\ma+\mb). $$
\end{Lem}
\begin{proof} By Lemma~\ref{lemma2} is suffices to find conditions
on $\xi$ which imply that $z_{\a}(\xi,\zeta)\in E(\Phi,\ma+\mb)$ for
each root $\a\in\Phi$ and $\zeta\in R$.
\par\smallskip
{\bf General case.}
First, assume that $\al$ is short or $\Phi\neq\C_l$. By Lemma~\ref{lem1}
there exist roots $\beta$ and $\gamma$ such that $\beta+\gamma=\alpha$
and $N_{\beta\gamma 11}=1$. In this case we prove that
$z_{\a}(ab,\zeta)\in E(\Phi,\ma+\mb)$ for each root $\a\in\Phi$ and all
$a\in\ma$, $b\in\mb$ and $\zeta\in R$. With this end we decompose
$x_{\alpha}(ab)$ as follows:
$$ x_{\alpha}(ab)= [x_{\beta}(a),x_{\gamma}(b)]
\prod x_{i\beta+j\gamma}(-N_{\gamma\delta ij}a^ib^j), $$
\noindent
where the product on the right hand side is taken over all roots
$i\beta+j\gamma\neq\alpha$. Conjugating this equality by $x_{-\a}(\zeta)$,
we obtain an expression of $z_{\a}(ab,\zeta)$ as a product
of elementary root unipotents belonging either to $E(\Phi,\ma)$ or
to $E(\Phi,\mb)$, or, as in the case of factors occuring in the tail,
even to $E(\Phi,\ma\mb)$.
\par\smallskip
{\bf Case $\Phi=\C_l$.}
This leaves us with the analysis of the exceptional case, where
$\Phi=\C_l$ and the root $\a$ is long. We will have to use several
instances of the Chevalley commutator formula.
\par
First of all, there exist {\it short\/} roots $\beta$ and $\gamma$ such
that $\beta+\gamma=\alpha$ and $N_{\beta\gamma 11}=2$. Thus,
$$ x_{\alpha}(2ab)= [x_{\beta}(a),x_{\gamma}(b)], $$
\noindent
for all $a\in\ma$ and $b\in\mb$. Now, exactly the same argument, as in the
general case, shows that $z_{\a}(2ab,\zeta)\in E(\Phi,\ma+\mb)$. This
shows that when $2\in R^*$ we again recover the general answer.
\par
By Lemma~\ref{lem1} there exist a long root root $\beta$ and a short
root $\gamma$ such that $\beta+2\delta=\alpha$ and $N_{\beta\gamma 12}=\pm1$.
Without loss of generality we can assume that $N_{\beta\gamma 12}=\pm1$,
otherwise we would just replace the $x_{\gamma}(a)$ in the following
formula by $x_{\gamma}(-a)$. We decompose $x_{\beta}(s^ht^ma)$ as follows:
$$ x_{\beta}(ab^2)=[x_{\gamma}(a),x_{\delta}(b)]
x_{\gamma+\delta}(-N_{\gamma\delta 11}ab), $$
\noindent
or all $a\in\ma$ and $b\in\mb$. Now, exactly the same argument, as in the
general case, shows that $z_{\a}(ab^2,\zeta)\in E(\Phi,\ma+\mb)$.
\par
Interchanging $a$ and $b$ in the above formula, we see, that
$z_{\a}(ab^2,\zeta)\in E(\Phi,\ma+\mb)$. To finish the proof, it remains
only to refer to the preceding lemma.
\end{proof}

In the next lemma we calculate the {\it lower\/} level of the mixed
commutator subgroup.
\begin{Lem}\label{nlln}
Let\/ $\rk(\Phi)\ge 2$. In the cases\/ $\Phi=\C_2,\G_2$ assume
that $R$ does not have residue fields\/ ${\Bbb F}_{\!2}$ of\/ $2$ elements and
in the case\/ $\Phi=\C_l$, $l\ge 2$, assume additionally that
any\/ $c\in R$ is contained in the ideal\/ $c^2R+2cR$.
\par
Then for any two ideals\/ $\ma$ and\/ $\mb$
of the ring\/ $R$ one has the following inclusion
$$ E(\Phi,R,\ma\mb)\le\big [E(\Phi,R,\ma),E(\Phi,R,\mb)\big]. $$
\end{Lem}
\begin{proof}
The mixed commutator of two normal subgroups is normal. Thus, it
suffices to prove that
$$ E(\Phi,\ma\mb)\le\big[E(\Phi,R,\ma),E(\Phi,R,\mb)\big], $$
\noindent
and the result will automatically follow. Actually, we prove a
slightly stronger inclusion
$E(\Phi,\ma\mb)\le\big[E(\Phi,\ma),E(\Phi,\mb)\big]$.
This more precise formula is not used in the present paper, but
it still might be useful to record this for future applications.
Denote $H=\big[E(\Phi,\ma),E(\Phi,\mb)\big]$. Then our claim amounts
to the following. Let $\a\in\Phi$, $a\in\ma$ and $b\in\mb$. Then
$x_{\a}(ab)\in H$.
\par\smallskip
$\bullet$ First, assume that $\a$ can be embedded in a root
system of type $\A_2$. Then there exist roots $\beta,\gamma\in\Phi$,
of the same length as $\a$ such that $\a=\beta+\gamma$, and
$N_{\beta\gamma 11}=1$. Then
$$ [x_{\beta}(a),x_{\gamma}(b)]=x_{\a}(ab)\in H. $$
\noindent
This proves the lemma for simply laced Chevalley groups, and for the
Chevalley group of type $\F_4$. It also proves necessary inclusions for
{\it short\/} roots in Chevalley groups of type $\C_l$, $l\ge 3$, and
for {\it long\/} roots in Chevalley groups of type $\B_l$, $l\ge 3$,
and of type $\G_2$.
\par\smallskip
$\bullet$ Next, assume that $\a$ can be imbedded in a root
system of type $\C_2$ as a {\it long\/} root. We wish to prove that
$x_{\a}(ab)\in H$, where $a\in\ma$ and $b\in\mb$. As the first
approximation, we prove that $x_{\a}(a^2b),x_{\a}(ab^2)\in H$.
\noindent
There exist roots $\beta,\gamma\in\Phi$, such that $\a=\beta+2\gamma$
and $N_{\beta\gamma 11}=1$. Clearly, in this case $\beta$ is long and
$\gamma$ is short. Take an arbitrary $c\in R$. Then
$$ [x_{\beta}(ca),x_{\gamma}(b)]=
x_{\beta+\gamma}(cab)x_{\a}(\pm cab^2)\in H, $$
\noindent
whereas
$$ [x_{\beta}(a),x_{\gamma}(cb)]=
x_{\beta+\gamma}(cab)x_{\a}(\pm c^2ab^2)\in H. $$
\noindent
Comparing these two inclusions we can conclude that
$x_{\a}(\pm (c^2-c)ab^2)\in H$.
\noindent
Since by assumption $R$ does not have residue field of $2$ elements,
the ideal generated by $c^2-c$, where $c\in R$, is not contained in any
maximal ideal, and thus coincides with $R$. It follows that
$x_{\a}(ab^2)\in H$. Interchanging $a$ and $b$ we see that
$x_{\a}(a^2b)\in H$.
\par\smallskip
$\bullet$ Now, assume that $\a$ can be imbedded in a root system of
type $\C_2$ as a {\it short\/} root. Choose the same $\beta$ and $\gamma$
as in the preceding item. In other words, $\a=\beta+\gamma$, $\beta$ is
long, $\gamma$ is short, and $N_{\beta\gamma 11}=1$. Then
$$ [x_{\beta}(a),x_{\gamma}(b)]=x_{\alpha}(ab)x_{\a+\gamma}(\pm ab^2). $$
\noindent
From the previous item we already know that the second factor belongs
to $H$ {\it provided\/} that $R$ does not have residue field of $2$
elements. Actually, from the first item, we know that for $\Phi=\B_l$,
$l\ge 3$, even the stronger inclusion $x_{\a+\gamma}(\pm ab)\in H$
holds without any such assumption.
\par
Thus, in both cases we can conclude that $x_{\alpha}(ab)\in H$, 
for a {\it short\/} root $\alpha$. Again, already from the first item 
we know that for $\Phi=\C_l$, $l\ge 3$, this inclusion holds without 
any assumptions on $R$.
\par\smallskip
On the other hand, a {\it long\/} root $\alpha$ of $\Phi=\C_l$,
$l\ge 3$, cannot be embedded in an irreducible rank 2 subsystem 
other than $\C_2$. This leaves us with the analysis of exactly two 
rank 2 cases: $\Phi=\C_2$ and $\a$ is long and $\Phi=\G_2$ and $\a$ 
is short. This is where one needs the additional assumptions on $R$.
\par\smallskip
$\bullet$ Let $\Phi=\C_2$ and $\a$ is long. Then $\a$ can be expressed
as $\a=\beta+\gamma$ for two {\it short\/} roots $\beta$ and $\gamma$.
Interchanging $\beta$ and $\gamma$ we can assume that
$N_{\beta\gamma 11}=2$. Then one has
$$ [x_{\beta}(a),x_{\gamma}(b)]=x_{\a}(2ab)\in H. $$
\noindent
One the other hand, we already know that $x_{\a}(a^2b)\in H$.
Since by assumption the ideal generated by $2a$ and $a^2$
contains $a$, we can conclude that $x_{\a}(ab)\in H$.
\par\smallskip
$\bullet$ Finally, let $\Phi=\G_2$ and $\a$ is short. We wish to
prove that $x_{\a}(ab)\in H$, where $a\in\ma$ and $b\in\mb$.
With this end we argue in the same way as for the case of $\Phi=\C_2$.
Actually, now it is even easier, since we already have necessary
inclusions for {\it long\/} roots.
\par
Again, as the first approximation, we prove that
$x_{\a}(a^2b),x_{\a}(ab^2)\in H$. With this end, express $\a$
as $\a=\beta+2\gamma$, where $\beta$ is short, $\gamma$ is long,
and $N_{\beta\gamma 11}=1$. Take an arbitrary $c\in R$. Then
$$ [x_{\beta}(ca),x_{\gamma}(b)]=
x_{\a}(cab)x_{\alpha+\beta}(\pm c^2a^2b)
x_{3\beta+\gamma}(\pm c^3a^3b)x_{3\beta+2\gamma}(\pm c^3a^3b^2)\in H, $$
\noindent
whereas
$$ [x_{\beta}(a),x_{\gamma}(cb)]=
x_{\a}(cab)x_{\alpha+\beta}(\pm ca^2b)
x_{3\beta+\gamma}(\pm ca^3b)x_{3\beta+2\gamma}(\pm c^2a^3b^2)\in H. $$
\noindent
Since the roots $3\beta+\gamma$ and $3\beta+2\gamma$ are long, from
the first item we already know that the corresponding root elements
belong to $H$. Thus,
$$ x_{\a}(cab)x_{\alpha+\beta}(\pm c^2a^2b),
x_{\a}(cab)x_{\alpha+\beta}(\pm ca^2b)\in H. $$
\noindent
Comparing these inclusions, we conclude that
$x_{\alpha+\beta}(\pm (c^2-c)a^2b)\in H$ for all $c\in R$. Again, since
$R$ does not have residue field of two elements, it follows that
$x_{\alpha+\beta}(\pm a^2b)\in H$. Interchanging $a$ and $b$,
we see that $x_{\a+\beta}(ab^2)\in H$.
\par
It only remains to look at the commutator
$$ [x_{\beta}(a),x_{\gamma}(b)]=
x_{\a}(ab)x_{\alpha+\beta}(\pm a^2b)
x_{3\beta+\gamma}(\pm a^3b)x_{3\beta+2\gamma}(\pm a^3b^2)\in H. $$
\noindent
Since all elementary factors on the right hand side, apart from the
first one, already belong to $H$, we can conclude that this first
factor also belongs to $H$, in other words, $x_{\a}(ab)\in H$, as
claimed
\end{proof}
Not to overburden the present paper with technical details, here
we only consider the usual relative subgroups depending on one
parameter. To illustrate, why we do this, let us state a general
version of Lemma~17, with form parameters, which can be easily
derived from the proof of Lemma 17.
\begin{Lem}\label{nllnbis}
Let\/ $\rk(\Phi)\ge 2$. Then for any two for ideals\/ $\ma$ and\/ $\mb$
of the ring\/ $R$ one has the following inclusions
$$ E(\Phi,R,\ma\mb,i_{\Phi}\ma\mb+\ma\supb{2}\md+\mb\supb{2}\mc)\le
\big [E(\Phi,R,\ma,\mc),E(\Phi,R,\mb,\md)\big]. $$
\end{Lem}
Without the additional assumption in the case $\C_l$, $l\ge 2$, the
upper and lower levels of the commutator of two relative elementary
subgroups do not coincide, and $E(\Phi,R,\ma\mb)$ in the statement
of Lemma~17 should be replaced\footnote{After the submission of the 
present paper, Himanee Apte and Alexei Stepanov \cite{AS} addressed
similar problems from a slightly different viewpoint. Their approach
depends on similar level calculations, and in particular, they 
indicate missing assumptions in previous publications, and 
provide detailed proofs for the case of $\Phi=\C_l$, without such
additional assumptions, see in particular, \cite{AS}, Lemma 5.2.} 
by $E(\Phi,R,\ma\supb{2}\mb+2\ma\mb+\ma\mb\supb{2})$.
Nevertheless, when $\ma$ and $\mb$ are comaximal, $\ma+\mb=R$,
these levels do coincide, so that no additional assumption is
necessary in the statement of Theorem 3.
\par
Next lemma bounds the upper level of mixed commutator subgroups.
Observe, that it also holds for $\rk(\Phi)=1$.
\begin{Lem}\label{GUGU}
Let\/ $\rk(\Phi)\ge 1$. Then for any two ideals\/ $\ma$ and\/ $\mb$
of the ring\/ $R$ one has the following inclusion
$$ \big[G(\Phi,R,\ma),C(\Phi,R,\mb)\big]\le G(\Phi,R,\ma\mb). $$
\end{Lem}
\begin{proof} Consider a faithful rational representation
$G(\Phi,R)\le\GL(n,R)$. Then
$$ G(\Phi,R,\ma)\le\GL(n,R,\ma),\qquad C(\Phi,R,\mb)\le C(n,R,\mb). $$
\noindent
Now, by lemma 5 of \cite{VS10} one has
$$ \big[G(\Phi,R,\ma),C(\Phi,R,\mb)\big]\le\big[\GL(n,R,\ma),C(n,R,\mb)\big]\le
\GL(n,R,\ma\mb). $$
\noindent
Since the left hand side is a subgroup of $G(\Phi,R)$, by
lemma~\ref{lemma6} we get
$$\big [G(\Phi,R,\ma),C(\Phi,R,\mb)\big]\le G(\Phi,R)\cap\GL(n,R,\ma\mb)
=G(\Phi,R,\ma\mb). $$
\end{proof}


\section{Relative conjugation calculus}

This section, and the next one constitute the technical core of the
paper. Here, we develop a relative version of the conjugation calculus
in Chevalley groups, whereas in the next section we evolve a relative
version of the commutator calculus. Throughout this section we assume
$\rk(\Phi)\ge 2$.
\par
In our survey \cite{RN} we explain the essence of Bak's method in
non-technical terms, and in our conference paper \cite{yoga}, joint
with Alexei Stepanov, we discuss the general philosophy of our versions
of that method, their interpretation in terms of $s$-adic topologies,
and their precise relation with other localisation methods. Not to
repeat ourselves, we simply refer the reader to these two sources,
and the references therein.
\par
For future applications we allow {\it two\/} localisation parameters.
Strictly speaking, this is not necessary for the proof of Theorem 1.
However, this is essential in the proof of Theorem 2 and is an absolute
must for the more advanced applications we ultimately have in mind,
such as {\it general\/} multiple commutator formulas, where
{\it none\/} of the factors is elementary. With this end we fix
{\it two\/} elements $s,t\in R$ and look at the localisation
$$ R_{st}=(R_s)_t=(R_t)_s. $$
\noindent
All calculations in this and the next sections take place in
$E(\Phi,R_{st})$. Thus, when we write something like
$E(\Phi,s^pt^qR)$, or $x_{\a}(s^pa)$, what we {\it really\/}
mean, is $E\big(\Phi,F_s(s^pt^qR)\big)$, or $x_{\a}(F_s(s^pa))$,
respectively, but we suppress $F_s$ in our notation.
This shouldn't lead to a confusion, since here we
{\it never\/} refer to elements or subgroups of $G(\Phi,R)$.
\par
The overall strategy in this and the next sections is {\it exactly\/}
the same, as in the proofs of Lemmas 3.1 and 4.1 of \cite{RN1} and
in the proofs of Lemmas 8--10 of \cite{SV10}. In turn, as we have
already mentioned in the introduction, both \cite{RN1} and \cite{SV10}
followed the general scheme proposed by Anthony Bak \cite{B4} for the
general linear group, and developed by the first author \cite{RH,RH2}
for unitary groups. Ideologically closely congnate, the actual
calculations in \cite{RN1,SV10} were technically noticeably different
from those in \cite{B4,RH,RH2} in some respects, due to the two
contrasting factors: some important technical simplifications, and
the fancier forms of the Chevalley commutator formula.
\par
However, now we wish to do the same at the {\it relative\/}, rather
than absolute level. In other words, we have to introduce another
parameter belonging to an ideal $\ma\trianglelefteq R$. The difference
with the existing versions of localisation is that whereas powers of
localising elements $s$ and $t$ are at our disposal, and can be
distributed among the factors, the ideal $\ma$ is fixed, and cannot
be distributed.
\par
The first main objective of the conjugation calculus is to establish
that conjugation by a fixed matrix $g\in G(\Phi,R_s)$ is continuous
in $s$-adic topology. In the proof one uses a base of neighborhoods
of $e$ and establishes that for any such neighborhood $V$ there
exists another neighborhood $U$ such that ${}^gU\subseteq V$. Usually,
one takes either elementary subgroups $E(\Phi,s^k\ma)$ of level
$s^k\ma$, or {\it relative\/} elementary subgroups $E(\Phi,R,s^k\ma)$
of level $s^k\ma$, as a base.
\par
However, both choices are not fully satisfactory in that they lead to
extremelly onerous calculations. The reason is that the first of these
choices is too small as the neighbourhood on the right hand side,
while the second of these choices is too large as the neighbourhood
on the left hand side. The solution proposed for $\GL(n,R)$ in \cite{RZ}
and later applied to unitary groups in \cite{RNZ} consists in selecting
another base of neightborhoods
$$ E(\Phi,s^k\ma)\le E(\Phi,s^kR,s^k\ma)\le E(\Phi,R,s^k\ma), $$
\noindent
which is much better balanced with respect to conjugation. The following
definition embodies the gist of this method.
\begin{deff}
Let\/ $R$ be a commutative ring,\/ $\ma$ an ideal of\/ $R$
and\/ $s\in R$. For a positive integer\/ $k$, define
$$ E(\Phi,s^kR,s^k\ma)={E(\Phi,s^k\ma)}^{E(\Phi,s^kR)} $$
\noindent
as the normal closure of $E(\Phi,s^k\ma)$ in $E(\Phi,s^kR)$, i.e., the
group generated by ${}^e x_\alpha(s^ka)$ where $e\in E^K(\Phi,s^kR)$,
for some positive integer $K$, $a\in\ma$ and $\alpha\in\Phi$.
\end{deff}
The following lemma is a relative version of Lemma 3.1 of \cite{RN1}
and of Lemma 8 of \cite{SV10}. Observe, that we could not simply
put $E(\Phi,s^pt^q\ma)$ on the right hand side. While the powers of
$s$ and $t$ can be distributed among the factors on the right hand side
in the calculations below, this is not the case for the parameter
$a\in\ma$. This is why we need conjugates by elements of
$E(\Phi,s^pt^qR)$.
\par
Observe, that the proof works in terms of roots alone, and thus one
gets {\it uniform\/} estimates for the powers of $s$ and $t$, which
do not depend on the ideal $\ma$. This circumstance, the
{\it equi-continuity\/} of conjugation by $g\in G(\Phi,R_s)$ on
congruence subgroups, is extremely important, and will be repeatedly
used in the sequel.
\begin{Lem}\label{lem2}
If\/ $p$, $q$ and\/ $k$ are given, there exist\/ $h$ and\/ $m$ such that
$$ {}^{\textstyle{E^1\left(\Phi,\frac{1}{s^k}R\right)}}
E(\Phi,s^ht^m\ma)\subseteq E(\Phi,s^pt^qR,s^pt^q\ma). $$
\noindent
Such\/ $h$ and\/ $m$ depend on\/ $\Phi$, $k$, $p$ and $q$ alone,
but does not depend on the ideal\/ $\ma$.
\end{Lem}
\begin{proof}
Since by definition $E(\Phi,s^ht^m\ma)$ is generated by $x_{\beta}(s^ht^ma)$,
where $\beta\in\Phi$ and $a\in\ma$, it suffices to show that there exist
$h$ and $m$ such that
$$ {}^{ x_\al\left(\frac{r}{s^k}\right)}
x_{\beta}(s^ht^ma)\in E(\Phi,s^pt^qR,s^pt^q\ma), $$
\noindent
for any $x_\al(r/s^k)\in E^1\big(\Phi,\frac{1}{s^k}R\big)$ and any
$x_{\beta}(s^ht^ma)\in E(\Phi,s^ht^m\ma)$.

\smallskip
\noindent
{\bf Case 1.}  Let $\al\neq-\beta$ and set $h\geq i_{\Phi}k+p+1$, $m\ge q$.
By the Chevalley commutator formula,
$$ x_{\al}\Big(\frac{r}{s^k}\Big)x_{\beta}(s^ht^ma)x_{\al}\Big(-\frac{r}{s^k}\Big)=
\prod_{i\al+j\beta\in\Phi}x_{i\al+j\beta}
\bigg(N_{\al\beta ij}{\Big(\frac{r}{s^k}\Big)}^i{(s^ht^ma)}^j\bigg)x_{\beta}(s^ht^ma) $$
\noindent
and a quick inspection shows that the right hand side of the above
equality is in $E^L(\Phi,s^pt^q\ma)$, where $L=2,3$ or $5$,
depending on whether $\Phi$ is simply laced, doubly laced or
triply laced. Clearly,
$$ E^L(\Phi,s^pt^q\ma)\subseteq E(\Phi,s^pt^q\ma)\le E(\Phi,s^pt^qR,s^pt^q\ma). $$
\smallskip
\noindent
{\bf Case 2.} Let $\al=-\beta$ and one of the following holds: $\beta$
is short or $\Phi\neq\C_l$. By Lemma~\ref{lem1} there exist roots $\gamma$
and $\delta$ such that $\gamma+\delta=\beta$ and $N_{\gamma\delta
11}=1$.  We set $h=2(i_{\Phi}k+p+1)$, $m=2q$, and decompose
$x_{\beta}(s^ht^ma)$ as follows:
$$ x_{\beta}(s^ht^ma)= [x_{\gamma}(s^{h/2}t^{m/2}),x_{\delta}(s^{h/2}t^{m/2}b)]
\prod x_{i\gamma+j\delta}(-N_{\gamma\delta ij}(s^{h/2}t^{m/2})^i(s^{h/2}t^{m/2}a)^j), $$
\noindent
where the product on the right hand side is taken over all roots
$i\gamma+j\delta\neq\beta$.
\par
Conjugating this expression by $x_\al\big(\frac{r}{s^k}\big)$ we get
\begin{multline*}
{}^{x_{\al}\left(\frac{r}{s^k}\right)}x_{\beta}(s^ht^ma)=
\Big[{}^{x_{\al}\left(\frac{r}{s^k}\right)}x_{\gamma}(s^{h/2}t^{m/2}),
{}^{x_{\al}\left(\frac{r}{s^k}\right)}x_{\delta}(s^{h/2}t^{m/2}a)\Big]\cdot\\
\prod{}^{x_{\al}\left(\frac{r}{s^k}\right)}
x_{i\gamma+j\delta}
\Big(-N_{\gamma\delta ij}(s^{h/2}t^{m/2})^i(s^{h/2}t^{m/2}a)^j\Big).
\end{multline*}
\noindent
Obviously, $\gamma,\delta$ and all the roots $i\gamma+j\delta\neq\beta$,
occuring in the product, are distinct from $-\al$. Now, by Case~1 the
first element of the commutator belongs to $E(\Phi,s^pt^qR)$, while the
second element of the commutator, and all factors of the product belong
to $E(\Phi,s^pt^q\ma)$. Since $E(\Phi,s^pt^qR,s^pt^q\ma)$ is normalised
by $E(\Phi,s^pt^qR)$, it follows that each term on right hand side sits
in $E(\Phi,s^pt^qR,s^pt^q\ma)$.

\smallskip
\noindent
{\bf Case 3.} Let $\Phi=\C_l$ and $\al=-\beta$ be a long root. By
Lemma~\ref{lem1} there exist roots $\gamma$ and $\delta$ such that either
$\gamma+2\delta=\beta$ and $N_{\gamma\delta 12}=1$, or
$2\gamma+\delta=\beta$ and $N_{\gamma\delta21}=1$.  We look at the
first case, the second case is similar. Alternatively, if
$N_{\gamma\delta 12}=-1$, one could change the sign of $x_{\gamma}(a)$
in the following formula by $x_{\gamma}(-a)$.  We set $h=3(i_{\Phi}k+p+1)$
and $m=3q$, and decompose $x_{\beta}(s^ht^ma)$ as follows:
$$ x_{\beta}(s^ht^ma)=
\big[x_{\gamma}(s^{h/3}t^{m/3}a),x_{\delta}(s^{h/3}t^{m/3})\big]\cdot
x_{\gamma+\delta}(-N_{\gamma\delta 11}s^{2h/3}t^{2m/3}a), $$
\noindent
Conjugating this expression by $x_\al\big(\frac{r}{s^k}\big)$ we get
\begin{multline*}
{}^{x_{\al}\left(\frac{r}{s^k}\right)}x_{\beta}(s^ht^ma)=
\Big[{}^{x_{\al}\left(\frac{r}{s^k}\right)}x_{\gamma}(s^{h/3}t^{m/3}a),
{}^{x_{\al}\left(\frac{r}{s^k}\right)}x_{\delta}(s^{h/3}t^{m/3})\Big]\cdot\\
{}^{x_{\al}\left(\frac{r}{s^k}\right)}
x_{\gamma+\delta}(-N_{\gamma\delta 11}s^{2h/3}t^{2m/3}a).
\end{multline*}
\noindent
As in Case 2, we can apply Case~1 to each conjugate on the right hand
side. The first element of the commutator, and the last factor belong
to $E(\Phi,s^pt^q\ma)$, while the second element of the commutator
belongs to $E(\Phi,s^pt^qR)$. Again, it remains only to recall that
$E(\Phi,s^pt^qR,s^pt^q\ma)$ is normalised by $E(\Phi,s^pt^qR)$.
\end{proof}

Now, the trick is that the elementary group $E(\Phi,s^ht^m\ma)$
on the left hand side can be {\it effortlessly\/} replaced by
$E(\Phi,s^ht^mR,s^ht^m\ma)$. Notice, that this step does not work
like that for the {\it usual\/} relative group $E(\Phi,R,s^ht^m\ma)$.
The reason are the obstinate denominators in the exponent, which
force to reiterate the procedure several times, according to the
length of the conjugating element.
\begin{Lem}\label{lem3}
If\/ $p$, $q$ and\/ $k$ are given, there exist $h$
and\/ $m$ such that
$$ {}^{E^1\left(\Phi,\frac{1}{s^k}R\right)}
E(\Phi,s^ht^mR,s^ht^m\ma)\subseteq E(\Phi,s^pt^qR,s^pt^q\ma). $$
\end{Lem}
\begin{proof}
Indeed, one has ${\vphantom{\big(}}^h\big({}^gx\big)={}^{(hgh^{-1})h}x$.
Thus,
\begin{multline*}
{}^{E^1\left(\Phi,\frac{R}{s^k}\right)}E(\Phi,s^ht^mR,s^ht^m\ma)=
{\vphantom{\Big(}}^{E^1\left(\Phi,\frac{R}{s^k}\right)}
\Big({}^{E(\Phi,s^ht^mR)}E(\Phi,s^ht^mR)\Big)=\\
={\vphantom{\Big(}}^{{}^{E^1\left(\Phi,\frac{R}{s^k}\right)}E(\Phi,s^ht^mR)}
\Big({}^{E^1\left(\Phi,\frac{R}{s^k}\right)}E(\Phi,s^ht^mR)\Big).
\end{multline*}
Now, by the preceding lemma, for any given $p$ and $q$ there exist
sufficiently large $h$ and $m$ such that the exponent is contained
in $E(\Phi,s^pt^qR)$, while the base is contained in
$E(\Phi,s^pt^qR,s^pt^q\ma)$. It remains to recall that by the
very definition $E(\Phi,s^pt^qR,s^pt^q\ma)$ is normalised by
$E(\Phi,s^pt^qR)$.
\end{proof}
Now, since $h$ and $m$ in Lemma~\ref{lem2} do not depend on the ideal
$\ma$, the preceding lemma immediately implies the following fact.
\begin{Lem}\label{lat}
If\/ $p,k$ are given, then there is an\/ $q$ such that
$$ {}^{E^1\left(\Phi,\frac{R}{s^k}\right)} \big
[E(\Phi,s^qR,s^q\ma), E(\Phi,s^qR,s^q\mb)\big] \subseteq
\big [E(\Phi,s^pR,s^p\ma), E(\Phi,s^pR,s^p\mb)\big]. $$
\end{Lem}
Iterated application of the above lemma, gives the following result.
\begin{Lem}\label{lat-bis}
If\/ $p,k$ and\/ $L$ are given, then there is an\/ $q$ such
that $$ {}^{E^L\left(\Phi,\frac{R}{s^k}\right)} \big
[E(\Phi,s^qR,s^q\ma), E(\Phi,s^qR,s^q\mb)\big] \subseteq
\big [E(\Phi,s^pR,s^p\ma), E(\Phi,s^pR,s^p\mb)\big]. $$
\end{Lem}
Now, we are all set for the next round of calculations. Namely,
it is our intention to obtain similar formulae, admitting
denominators not only in the exponent, but also on the ground
level.


\section{Relative commutator calculus}

To implement second localisation, we will have to be able to fight
powers of {\it two\/} elements in the denominator. The relative
commutator calculus turns out to be much more technically demanding,
than the relative conjugation calculus. Not only that the first step
of induction is {\it by far\/} the hardest one. Actually, even the
usually trivial first {\it substep\/} of the first step, the case
of two non-opposite roots, turns out to be a real challenge. As
always, it is extremely important for the sequel that the resulting
power estimates do not depend on the ideals $\ma$ and $\mb$.
\par
Throughout we continue to assume $\rk(\Phi)\ge 2$. In the cases
$\Phi=\B_2=\C_2$ and $\Phi=\G_2$ we additionally assume that $2\in R^*$.
These standing assumptions will not be repeated in the statements
of lemmas.
\begin{Lem}\label{dfle}
If\/ $p,q,k,m$ are given, then there exist\/ $l$ and\/ $n$ such that
$$ \Big [E^1\Big(\Phi,\frac{t^l}{s^k}\ma\Big),
E^1\Big(\Phi,\frac{s^n}{t^m}\mb\Big)\Big]
\subseteq \big[E(\Phi,s^pt^qR,s^pt^q\ma),E(\Phi,s^pt^qR,s^pt^q\mb)\big]. $$
\noindent
These\/ $l$ and\/ $n$ depend on\/ $\Phi,p,q,k,m$ alone, and do
not depend on the choice of ideals\/ $\ma$ and\/ $\mb$.
\end{Lem}
\begin{proof} Let $\a,\b\in\Phi$, $a\in\ma$ and $b\in\mb$.
We have to prove that
$$ \bigg[x_\alpha\Big(\frac{t^l}{s^k}a\Big),
x_\beta\Big(\frac{s^n}{t^m}b\Big)\bigg]\in
\big[E(\Phi,s^pt^qR,s^pt^q\ma),E(\Phi,s^pt^qR,s^pt^q\mb)\big]. $$
\noindent
The partition into cases is exactly the same as in the proof of
Lemma~\ref{lem2}, but the calculations themselves --- and the
resulting length bounds, should we attempt to record them ---
are now much fancier.
\par
\smallskip \noindent
{\bf Case 1.}  Let $\al\neq-\beta$. Then using 
the Chevalley commutator formula we get
\begin{multline*}
\biggl[x_\a\Bigl(\frac{t^l}{s^k}a\Bigr),
x_{\beta}\Bigl(\frac{s^{n}}{t^{m}}b\Bigr)\biggr]=
\prod_{i,j>0}x_{i\a+j\beta}\biggl(N_{\alpha\beta ij}
\Bigl(\frac{t^l}{s^k}a\Bigr)^i
\Bigl(\frac{s^{n}}{t^{m}}b\Bigr)^{\hskip-2pt j\ }\biggr)=\\
=\prod_{i\alpha+j\beta\in\Phi}
x_{i\alpha+j\beta}\Big(N_{\alpha\beta ij}
s^{jn-ik}t^{il-jm}a^ib^j\Big).
\end{multline*}
\noindent
Clearly, one can take sufficiently large $l$ and $n$.
It suffices to show, that taking large enough $l$ and $n$ we can
redistribute powers of $s$ and $t$ between the first and the second
parameters in each factor on the right hand side in such a way,
that the resulting product can be expressed as a {\it product\/}
of commutators without denominators.
\par
\smallskip\noindent
{\bf Warning.} This is one of the key new points in the whole argument,
where we cannot thoughtlessly imitate \cite{RN1} or \cite{SV10}. Namely,
expressing an element as a product of commutators without denomitators,
with parameters sitting where they should, is not quite the same as
just observing that taking large enough $l$ and $n$ we can kill all
the denominators in each factor on the right hand side of the Chevalley
commutator formula. This is precisely the point, where the cases
$\Phi=\C_2,\G_2$ require {\it substantial\/} extra care.
\par\smallskip
$\bullet$ First, assume that the right hand side of the Chevalley
commutator formula consists of one factor. In this case
$$ \biggl[x_\a\Bigl(\frac{t^l}{s^k}a\Bigr),
x_{\beta}\Bigl(\frac{s^{n}}{t^{m}}b\Bigr)\biggr]=
x_{\a+\b}(N_{\a\b}s^{n-k}t^{l-m}ab). $$
\noindent
Taking $n\ge 2p+k$ and $l\ge 2q+m$ we can rewrite this commutator as
a commutator without denominators as follows:
$$ \biggl[x_\a\Bigl(\frac{t^l}{s^k}a\Bigr),
x_{\beta}\Bigl(\frac{s^{n}}{t^{m}}b\Bigr)\biggr]=
\big[x_\a(s^pt^qa),x_{\beta}(s^{n-k-p}t^{l-m-q}b)\big]. $$
\noindent
Observe, that the right hand side belongs to
$\big[E(\Phi,s^pt^qR,s^pt^q\ma),E(\Phi,s^pt^qR,s^pt^q\mb)\big]$.
\par
The assumption of this item amounts to saying that $|\a|=|\b|$,
with the sole exception of two short roots in $\G_2$, whose sum
is a short root, where the right hand side of the Chevalley
commutator formula consists of {\it three\/} factors, rather
than one.
\par
Thus, in fact, we have established somewhat more, than claimed.
Namely, assume that if $\g=\a+\b$, $|\a|=|\b|$, and, moreover,
the mutual angle of $\a$ and $\b$ is not $2\pi/3$ if $\a,\b$
are short roots of $\Phi=\G_2$. Then for any $h\ge 2p$, any
$r\ge 2q$, any $a\in\ma$ and any $b\in\mb$ one has
$$ x_{\a+\b}(N_{\a\b}s^ht^rab)\in
\big[E(\Phi,s^pt^qR,s^pt^q\ma),E(\Phi,s^pt^qR,s^pt^q\mb)\big]. $$
\par
In particular, this proves Case 1 for {\it simply\/} laced systems.
\par\smallskip
$\bullet$ Actually, with the use of the above argument it is
easy to completely settle also the case of {\it doubly\/} laced
systems, except for $\Phi=\C_2$. Indeed, for doubly laced systems
it accounts for the case, where $|\a|=|\b|$. Now, let $\a$ and $\b$
have distinct lengths. If necessary, replacing $[x,y]$ by
$[y,x]=[x,y]^{-1}$, we can assume that $\a$ is long, and $\b$
is short. In this case
$$ y=\biggl[x_\a\Bigl(\frac{t^l}{s^k}a\Bigr),
x_{\beta}\Bigl(\frac{s^{n}}{t^{m}}b\Bigr)\biggr]=
x_{\a+\b}(N_{\a\b11}s^{n-k}t^{l-m}ab)
x_{\a+2\b}(N_{\a\b12}s^{2n-k}t^{l-2m}ab^2). $$
\par
Now, if $\Phi=\F_4$, every root embeds in a subsystem of type
$\A_2$. In other words, the root $\a+\b$ is a sum of two {\it short\/}
roots, whereas $\a+2\b$ is a sum of two {\it long\/} roots. Thus,\
taking $2n\ge 2p+k$ and $l\ge 2q+2m$, we see that each elementary
unipotent on the right hand side of the above formula is itself
a single commutator in
$\big[E(\Phi,s^pt^qR,s^pt^q\ma),E(\Phi,s^pt^qR,s^pt^q\mb)\big]$.
\par
The cases $\B_l$, $l\ge 3$ and $\C_l$, $l\ge 3$, are treated in
a similar way, and are only marginally trickier.
\par\smallskip
$\bullet$ First, let $\Phi=\B_l$, $l\ge 3$. In this case every
{\it long\/} root is a sum of two long roots. Clearly, the first
elementary factor in expression of the commutator
\begin{multline*}
z=\biggl[x_\a\Bigl({s^p}{t^{l-m-q}}a\Bigr),
x_{\beta}\Bigl({s^{n-k-p}}{t^{q}}b\Bigr)\biggr]=
x_{\a+\b}(N_{\a\b11}s^{n-k}t^{l-m}ab)\cdot\\
x_{\a+2\b}(N_{\a\b12}s^{2n-2k-p}t^{l-m+q}ab^2),
\end{multline*}
coincides with the first elementary factor of the above commutator.
If $l\ge 2q+m$ and $n\ge 2p+k$, one has
$z\in\big[E(\Phi,s^pt^qR,s^pt^q\ma),E(\Phi,s^pt^qR,s^pt^q\mb)\big]$.
On the other hand, if, moreover, $l\ge 2q+2m$, then the long
root unipotent
$$ yz^{-1}=x_{\a+2\b}(N_{\a\b12}(s^{2n-k}t^{l-2m}-s^{2n-2k-p}t^{l-m+q})ab^2) $$
\noindent
is also a single commutator in
$\big[E(\Phi,s^pt^qR,s^pt^q\ma),E(\Phi,s^pt^qR,s^pt^q\mb)\big]$,
by the first item.
\par\smallskip
$\bullet$ Now, let $\Phi=\C_l$, $l\ge 3$. In this case every
{\it short\/} root is a sum of two short roots. Set $p'=p$
if $p\equiv k\pmod 2$, and $p'=p+1$ otherwise. Then {\it second\/}
elementary factor in expression of the commutator
\begin{multline*}
z=\biggl[x_\a\Bigl({s^{p'}}{t^{l-2m-2q}}a\Bigr),
x_{\beta}\Bigl({s^{(2n-k-p')/2}}{t^{q}}b\Bigr)\biggr]=
x_{\a+\b}(N_{\a\b11}s^{(2n-k+p')/2}t^{l-2m-q}ab)\cdot\\
x_{\a+2\b}(N_{\a\b12}s^{2n-k}t^{l-2m}ab^2),
\end{multline*}
coincides with the second elementary factor of the commutator $y$.
If $l\ge 3q+2m$ and $n\ge(2p+k+1)/2$, one has
$z\in\big[E(\Phi,s^pt^qR,s^pt^q\ma),E(\Phi,s^pt^qR,s^pt^q\mb)\big]$.
On the other hand, if, moreover, $n\ge (5p+k+1)/2$, then the short
root unipotent
$$ yz^{-1}=x_{\a+\b}(N_{\a\b11}(s^{n-k}t^{l-m}-s^{(2n-k+p')/2}t^{l-2m-q})ab) $$
\noindent
is also a single commutator in
$\big[E(\Phi,s^pt^qR,s^pt^q\ma),E(\Phi,s^pt^qR,s^pt^q\mb)\big]$,
by the first item.
\par\smallskip
$\bullet$ Finally, let $\Phi=\C_2$ or $\G_2$. We will see that
under assumption $2\in R^*$ the proof is essentially the same,
as in the above cases. First, let $\Phi=\C_2$, and let $\a,\b$,
$\a\neq\pm\b$, be two short roots. Then by the first item one
has
$$ x_{\a+\b}(2s^ht^rab)
\in \big[E(\Phi,s^pt^qR,s^pt^q\ma),E(\Phi,s^pt^qR,s^pt^q\mb)\big]. $$
\noindent
Since $2\in R^*$, it follows that $x_{\a+\b}(2s^ht^rab)$ is
a single commutator of requested shape, whenever $h\ge 2p$ and
$r\ge 2q$. Now, we can repeat exactly the same argument, as
in the case $\Phi=\B_l$, $l\ge 3$.
\par
Next, let $\Phi=\G_2$. First, observe that by the first item
$x_{\a}(s^ht^rab)$ is already a single commutator of the required
shape for any $h\ge 2p$ and any $r\ge 2q$. Now, let $\a,\b$ be
two short roots, whose sum is a short root. Then the Chevalley
commutator formula takes the following form
$$ [x_{\a}(\xi),x_{\b}(\zeta)]=
x_{\a+\b}(\pm 2\xi\zeta)
x_{2\a+\b}(\pm 3\xi^2\zeta)x_{\a+2\b}(\pm 3\xi\zeta^2), $$
\noindent
see, for example, \cite{steinberg67,carter1972} or \cite{vavplot}.
\par
Now, setting here $\xi=s^pt^{r-q}a$ and $\zeta=s^{h-p}t^qb$,
for some $a\in\ma$ and $b\in\mb$, we see that
\begin{multline*}
x_{\a+\b}(\pm 2s^ht^rab)
x_{2\a+\b}(\pm 3s^{h+p}t^{2r-q}a^2b)
x_{\a+2\b}(\pm 3s^{2h-p}t^{r+q}ab^2)\in\\
\big[E(\Phi,s^pt^qR,s^pt^q\ma),E(\Phi,s^pt^qR,s^pt^q\mb)\big],
\end{multline*}
\noindent
for any $h\ge 2p$ and $r\ge 2q$. Since each of the resulting
long root elements is already a single commutator of requested
shape, and $2\in R^*$, one sees that $x_{\a+\b}(s^ht^rab)$
is a product of at most three commutators of requested shape.
Now, we conclude the analysis of this case by exactly the
same argument, as in the case of $\Phi=\G_2$, and conclude
that for any two linearly independent roots, any $\a\in\ma$,
$b\in\mb$, and any $n\ge 2p+3k$, $l\ge 2q+3m$, one has
$$ \biggl[x_\a\Bigl(\frac{t^l}{s^k}a\Bigr),
x_{\beta}\Bigl(\frac{s^{n}}{t^{m}}b\Bigr)\biggr]\in
\big[E(\Phi,s^pt^qR,s^pt^q\ma),E(\Phi,s^pt^qR,s^pt^q\mb)\big], $$
\noindent
in fact, the commutator on the left hand side is the product
of not more than eight commutators of two elementary
unipotents, belonging to $E(\Phi,s^pt^qR,s^pt^q\ma)$ and
$E(\Phi,s^pt^qR,s^pt^q\mb)$, respectively.
\par\smallskip\noindent
{\bf Case 2.} Let $\alpha=-\beta$, and one of the following holds,
$\alpha$ is short or $\Phi\not=\C_l$. By Lemma~\ref{lem1}, there
are roots $\gamma$ and $\delta$ such that $\gamma+\delta=\alpha$
and $N_{\gamma\delta11}=1$. We can assume that
$k$ and $l$ are even and decompose
$x_\a\Bigl(\displaystyle{\frac{t^l}{s^k}a}\Bigr)$ as follows
\begin{equation}\label{decip}
x_\a\Bigl(\frac{t^l}{s^k}a\Bigr)=
\Bigl[x_\gamma\Bigl(\frac{t^{l/2}}{s^{k/2}}\Bigr),
x_\delta\Bigl(\frac{t^{l/2}}{s^{k/2}}a\Bigr)\Bigr]
\prod_{i\gamma+j\delta\in \Phi}x_{i\gamma+j\delta}
\Big(-N_{\gamma\delta ij}
\frac{t^{l(i+j)/2}}{s^{k(i+j)/2}}a^j\Big),
\end{equation}
\noindent
where the product on the right hand side is taken over all roots
$i\gamma+j\delta\not=\alpha$. Consider the commutator formula
\begin{equation}\label{myb}
\Big[[y,z]\prod_{i=1}^{t} u_i,x\Big]=
\prod_{i=1}^{t}{}^{[y,z]\prod_{j=1}^{i-1} u_j}\big[u_i,x\big]\big[[y,z],x\big]
\end{equation}
\noindent
where by convention $\prod_{j=1}^0 u_j=1$. Now let
$y=x_{\gamma}\Bigl(\displaystyle{\frac{t^{l/2}}{s^{k/2}}}\Bigr)$,
$z=x_{\delta}\Bigl(\displaystyle{\frac{t^{l/2}}{s^{k/2}}a}\Bigr)$
and $u_i$'s stand for the terms $x_{i\gamma+j\delta}(*)$
in Equation~\ref{decip}. Let
$x=x_{\beta}\Bigl(\displaystyle{\frac{s^{n}}{t^{m}}b}\Bigr)$
and plug these in to Equation~\ref{myb}. The terms $[u_i,x]$ are all
of the form considered in Case~1, and thus for suitable $l$ and $n$
they belong to
$\big[E(\Phi,s^pt^qR,s^pt^q\ma),E(\Phi,s^pt^qR,s^pt^q\mb)\big]$.
Thus, it immediately follows that
$\prod_{i=1}^{t}{}^{[y,z]\prod_{j=1}^{i-1} u_j}\big[u_i,x\big]$
belongs to this commutator group.
\par
We are left to show that for a suitable $q$
$$ \big[[y,z],x\big]=
\bigg[\Bigl[x_\gamma\Bigl(\frac{t^{l/2}}{s^{k/2}}\Bigr),
x_\delta\Bigl(\frac{t^{l/2}}{s^{k/2}}a\Bigr)\Bigr],
x_{\beta}\Bigl(\frac{s^{n}}{t^{m}}b\Bigr)\bigg] $$
\noindent
is also in $\big[E(\Phi,s^pt^qR,s^pt^q\ma),E(\Phi,s^pt^qR,s^pt^q\mb)\big]$.
Consider the conjugate
\begin{multline*}
{\phantom{\bigg[}}^{x_\gamma\Bigl(-\frac{t^{l/2}}{s^{k/2}}\Bigr)}
\bigg[\Bigl[x_\gamma\Bigl(\frac{t^{l/2}}{s^{k/2}}\Bigr),
x_\delta\Bigl(\frac{t^{l/2}}{s^{k/2}}a\Bigr)\Bigr],
x_{\beta}\Bigl(\frac{s^{n}}{t^{m}}b\Bigr)\bigg]
=\\
{\phantom{\bigg[}}^{x_\gamma\Bigl(-\frac{t^{l/2}}{s^{k/2}}\Bigr)}
\bigg[{\Bigl[x_\delta\Bigl(\frac{t^{l/2}}{s^{k/2}}a\Bigr),
x_\gamma\Bigl(\frac{t^{l/2}}{s^{k/2}}\Bigr)\Bigr]}^{-1},
x_{\beta}\Bigl(\frac{s^{n}}{t^{m}}b\Bigr)\bigg].
\end{multline*}
\noindent
By the Hall---Witt identity it can be rewritten as
\begin{multline*}
uv=
{\phantom{\bigg[}}^{x_\delta\Bigl(\frac{t^{l/2}}{s^{k/2}}a\Bigr)}
\bigg[x_\gamma\Bigl(-\frac{t^{l/2}}{s^{k/2}}\Bigr),
\Big[x_{\beta}\Bigl(\frac{s^{n}}{t^{m}}b\Bigr),
x_\delta\Bigl(-\frac{t^{l/2}}{s^{k/2}}a\Bigr)\Big]^{-1}\bigg]\cdot\\
{\phantom{\bigg[}}^{x_{\beta}\Bigl(\frac{s^{n}}{t^{m}}b\Bigr)}
\bigg[x_\delta\Bigl(\frac{t^{l/2}}{s^{k/2}}a\Bigr),
\Big[x_\gamma\Bigl(-\frac{t^{l/2}}{s^{k/2}}\Bigr),
x_{\beta}\Bigl(-\frac{s^{n}}{t^{m}}b\Bigr)\Big]^{-1}\bigg].
\end{multline*}
\noindent
Let us consider the factors separately.
\par
Since $\gamma,\delta\not=-\beta$, by Case~1 one can find
suitable $l$ and $n$ such that the commutator
$\Big[x_{\beta}\Bigl(\frac{s^{n}}{t^{m}}b\Bigr),
x_\delta\Bigl(-\frac{t^{l/2}}{s^{k/2}}a\Bigr)\Big]$ belongs to
$\big[E(\Phi,s^pt^qR,s^pt^q\ma),E(\Phi,s^pt^qR,s^pt^q\mb)\big]$,
and it immediately follows that $u$ belongs to this group.
\par
Applying Chevalley commutator formula to the internal
commutator in $v$, we have
$$ v={\phantom{\bigg[}}^{x_{\beta}\Bigl(\frac{s^{n}}{t^{m}}b\Bigr)}
\bigg[x_\delta\Bigl(\frac{t^{l/2}}{s^{k/2}}a\Bigr),
\prod_{i\gamma+j\beta\in\Phi}
x_{i\gamma+j\beta}\Big(-N_{\gamma\beta ij}
{\Bigl(-\frac{t^{l/2}}{s^{k/2}}\Bigr)}^i
{\Bigl(-\frac{s^{n}}{t^{m}}b\Bigr)}^j\Big)\bigg]. $$
\noindent
Now, for suitable $l$ and $n$ all
$x_{i\gamma+j\beta}\Bigl(\displaystyle{-N_{\gamma\beta ij}
{\Bigl(-\frac{t^{l/2}}{s^{k/2}}\Bigr)}^i
{\Bigl(-\frac{s^{n}}{t^{m}}b\Bigr)}^j}\Bigr)$
belong to $E(\Phi,s^{p'}t^{q'}\mb)$ for any prescribed $p'$ and $q'$.
Now employing Lemma~\ref{lem2} twice, we can secure that for
suitable $l$ and $n$ the second factor $v$ also belongs to the
commutator group
$\big[E(\Phi,s^pt^qR,s^pt^q\ma),E(\Phi,s^pt^qR,s^pt^q\mb)\big]$,
and we are done.

\smallskip
\noindent {\bf Case 3.} Let $\Phi=\C_l$ and $\alpha=-\beta$ be a long
root. Let $\Phi=\C_l$ and $\al=-\beta$ be a long root. By Lemma~\ref{lem1}
there exist roots $\gamma$ and $\delta$ such that either
$\gamma+2\delta=\beta$ and $N_{\gamma\delta 12}=1$, or
$2\gamma+\delta=\beta$ and $N_{\gamma\delta21}=1$. Like in the proof of
Lemma~\ref{lem2}, we lose nothing by looking at the {\it second\/}
case. Increasing $k$ and $l$, in necessary, we can assume that
$k$ and $l$ are divisible by 3 and decompose
$x_\a\Bigl(\displaystyle{\frac{t^l}{s^k}a}\Bigr)$ as follows
\begin{equation}\label{decip2}
x_\a\Bigl(\frac{t^l}{s^k}a\Bigr)=
\Bigl[x_\gamma\Bigl(\frac{t^{l/3}}{s^{k/3}}\Bigr),
x_\delta\Bigl(\frac{t^{l/3}}{s^{k/3}}a\Bigr)\Bigr]
\prod_{i\gamma+j\delta\in\Phi}x_{i\gamma+j\delta}
\Big(-N_{\gamma\delta ij}
\frac{t^{l(i+j)/3}}{s^{k(i+j)/3}}a^j\Big),
\end{equation}
\noindent
where the product is taken over all $(i,j)\neq(1,2)$. Now,
repeating the same argument as in Case 2, one can find suitable
$l$ and $n$ such that
$$ \bigg[x_\alpha\Big(\frac{t^l}{s^k}a\Big),
x_\beta\Big(\frac{s^n}{t^m}b\Big)\bigg]\in
\big[E(\Phi,s^pt^qR,s^pt^q\ma),E(\Phi,s^pt^qR,s^pt^q\mb)\big], $$
\noindent
as claimed.
\end{proof}

\begin{Lem}\label{dflebis}
If\/ $p,q,k,m$ and\/ $L$ are given, there exist\/ $l$ and\/ $n$,
independent of\/ $L$, such that
$$ \Big[E^L\Big(\Phi,\frac{t^l}{s^k}\ma\Big),
E^1\Big(\Phi,\frac{s^n}{t^m}\mb\Big)\Big]
\subseteq \big[E(\Phi,s^pt^qR,s^pt^q\ma),
E(\Phi,s^pt^qR,s^pt^q\mb)\big]. $$
\end{Lem}
\begin{proof}
An easy induction, using identity (C2), shows that
$$ \bigg[\prod_{i=1}^K u_i,x\bigg]=
\prod_{i=1}^K {}^{\prod_{j=1}^{K-i}u_j}\big[u_{K-i+1},x\big], $$
\noindent
where by convention $\prod_{j=1}^0 u_j=1$. This, with the fact that
$E(\Phi,s^pt^qR,s^{p}t^q\ma)$ and $E(\Phi,s^pt^qR,s^{p}t^q\mb)$
are both normalized by $E(\Phi,s^Pt^QR)$, where $P\geq p$, $Q\ge q$,
show that this lemma immediately follows from the previous one.
\end{proof}

Recall that
$\displaystyle{E\Big(\Phi,\frac{t^l}{s^k}R,\frac{t^l}{s^k}\ma\Big)}$
is generated by all elements of the form
$\displaystyle{{}^u x_\alpha\Bigl(\frac{t^l}{s^k}a\Bigr)}$,
where $u\in\displaystyle{E^L\Bigl(\Phi,\frac{t^l}{s^k}R\Bigr)}$,
for some $L$, and $a\in\ma$.
\noindent
\begin{Lem}
\label{lkjh}
If\/ $p,q,k,m$ are given, there exist\/ $l$ and\/ $n$ such that
$$ \Big[E\Big(\Phi,\frac{t^l}{s^k}R,\frac{t^l}{s^k}\ma\Big),
E^1\Big(\Phi,\frac{s^n}{t^m}\mb\Big)\Big]
\subseteq \big[E(\Phi,s^pt^qR,s^pt^q\ma),
E(\Phi,s^pt^qR,s^pt^q\mb)\big].  $$
\end{Lem}
\begin{proof}
Obviously, it suffices to prove that for any given $p,q,k,m$
and any $L$, there exist $l$ and $n$ independent of $L$ such
that
\begin{equation}\label{hghg}
\bigg[{\vphantom{\Big(}}^{E^L\left(\Phi,\frac{t^l}{s^k}R\right)}
E^1\Big(\Phi,\frac{t^l}{s^k}\ma\Big),
E^1\Big(\Phi,\frac{s^n}{t^m}\mb\Big)\bigg]
\subseteq \big[E(\Phi,s^pt^qR,s^pt^q\ma),
E(\Phi,s^pt^qR,s^pt^q\mb)\big],
\end{equation}
\noindent
after that the lemma follows from (\ref{hghg}) and identity (C2).
\par
Let $x\in\displaystyle{E^L\Big(\Phi,\frac{t^l}{s^k}R\Big)}$,
$y\in\displaystyle{E^1\Big(\Phi,\frac{t^l}{s^k}\ma\Big)}$
and $z\in\displaystyle{E^1\Big(\Phi,\frac{s^n}{t^m}\mb\Big)}$.
Using (C2) and the Hall---Witt identity we can write
\begin{multline*}
[{}^xy,z]=\big[y[y^{-1},x],z\big]=
{}^y\big[[y^{-1},x],z\big]\cdot[y,z]
={}^{yx^{-1}}\Bigl({}^{x}\bigl[[y^{-1},x],z\bigr]\Bigr)\cdot[y,z]=\\
{}^{yx^{-1}}\Bigl({}^{y}\bigl[x^{-1},[z,y]\bigr]
\cdot{}^{z^{-1}}\big[y^{-1},[x^{-1},z^{-1}]\bigr]\Bigr)\cdot[y,z].
\end{multline*}
\par
Now Lemma~\ref{dfle}, along with the fact that $E(\Phi,s^pt^qR,s^pt^q\ma)$
and $E(\Phi,s^pt^qR,s^pt^q\mb)$ are both normal in $E(\Phi,s^pt^qR)$,
imply that for suitable $l$ and $n$ all three commutators
$[y,z]$,  ${}^{y}\bigl[x^{-1},[z,y]\big]$ and
$\big[y^{-1},[x^{-1},z^{-1}]\big]$
are in
$$ \big [E(\Phi,s^pt^qR,s^pt^q\ma),E(\Phi,s^pt^qR,s^pt^q\mb)\big]. $$
\par
Now, we can invoke Lemma~\ref{lat} to ensure that there
are suitable $l$ and $n$ such that the conjugate
${}^{z^{-1}}\big[y^{-1},[x^{-1},z^{-1}]\big]$, and therefore the
whole commutator $[{}^xy,z]$, is in
$\big[E(\Phi,s^pt^qR,s^pt^q\ma),E(\Phi,s^pt^qR,s^pt^q\mb)\big]$.
\end{proof}


\section{Mixed commutator formula: localisation proof}\label{hgfdsa}

Now we are all set to complete a localisation proof of Theorem~\ref{main}.
In fact, we will prove a much more powerful result, in the spirit
of Theorem 5.3 of \cite{RN1}. We start with the following lemma,
whose proof mimics the proof of \cite{RN1}, Lemma~5.2, modulo
replacing elementary factors by commutators, and correcting
some misprints.
\begin{Lem}\label{justlem}
Fix an element\/ $s \in R$, $s\neq 0$. Then for any\/ $k$ and\/ $p$
there exists an\/ $r$ such that for any\/ $a\in\ma$, any\/
$g\in G(\Phi,R,s^r\mb)$ and any maximal ideal\/ $\mm$ of\/ $R$,
there exists an element\/ $t\in R\backslash\mm$, and an
integer\/ $l$ such that
\begin{equation}\label{bnm}
\bigg[x_{\a}\Big(\frac{t^l}{s^k}a\Big),F_s(g)\bigg]
\in\Big[E\big(\Phi,F_s(s^pR),F_s(s^p\ma)\big),
E\big(\Phi,F_s(s^pR),F_s(s^p\mb)\big)\Big].
\end{equation}
\end{Lem}
Note that here $q$ will depend on the choice of $x_\alpha$.
\begin{proof}
By 5.1 one has $G(\Phi,R)=\varinjlim G(\Phi,R_i)$, where the
limit is taken over all finitely generated subrings of $R$.
Thus, without loss of generality we may assume that $R$ is
Noetherian. To be specific, we can replace $R$ by the ring
generated by $a$, $s$ and the matrix entries of $g$ in a
faithful polynomial representation.
\par
Since $R_\mm$ is a local ring, by Lemma~\ref{abe-bis} we have the
decomposition
$$ G(\Phi,R_\mm,\mb_\mm)=E(\Phi,R_\mm,\mb_\mm)T(\Phi,R_\mm,\mb_\mm). $$
\noindent
Thus, one can decompose $F_\mm(g)$ as $F_\mm(g)=uh$ where
$u\in E(\Phi,R_\mm,\mb_\mm)\le G(\Phi,R_{\mm})$
and $h\in T(\Phi, R_\mm,\mb_\mm)$.
\par
Since $G(\Phi,R_M)=\varinjlim G(\Phi,R_t)$, over all $t\in
R\setminus M$, and the same holds for $E(\Phi,s^qR_M)$,
$T(\Phi,R_M,s^qR_M)$, etc., we can find an element $t\in
R\setminus M$ such that already $F_t(g)$ can be factored as
$F_t(g)=uh$, where $u\in E(\Phi,R_t,s^qR_t)$ and
$z\in T(\Phi,R_t,s^qR_t)$.
\par
On the other hand, since $R$ is assumed to be Noetherian,
$R_s$ is also Noetherian and by Lemma~\ref{noetherian} there exists
an $n$ such that the canonical homomorphism
$$ F_t:G(\Phi,R_s,t^nR_s)\map G(\Phi,R_{st}) $$
\noindent
is injective. Next, we take any $l>n$. Since
$x_{\al}\displaystyle{\Big(\frac{t^l}{s^k}a\Big)}\in
G(\Phi,R_s,t^n\ma_s)$,
and the principal congruence subgroup $G(\Phi,R_s,t^n\ma_s)$ is
normal in $G(\Phi,R_s)$, one has
$$ x=\Big[x_{\al}\Big(\frac{t^l}{s^k}a\Big),F_s(g)\Big]
\in G(\Phi,R_s,t^n\ma_s)\le G(\Phi,R_s,t^nR_s). $$
\noindent
Consider the image $F_t(x)\in G(\Phi,R_{st})$ of $x$ under
localisation with respect to $t$. Since $F_t$ is a homomorphism,
one has
$$ F_t(x)=\bigg[F_t\Big(x_{\a}\Big(\frac{t^l}{s^k}a\Big)\Big),F_{st}(g)\bigg]. $$
\noindent
Now $F_{st}(g)$ can be factored as
$F_{st}(g)=F_s(u)F_s(h)\in G(\Phi,R_{st})$. It follows that
\begin{eqnarray}
F_t(x)&=&\bigg[F_t\Big(x_{\al}\Big(
\frac{t^l}{s^k}a\Big)\Big),F_s(u)F_s(h)\bigg]= \nonumber\\
&&\bigg[F_t\Big(x_{\al}\Big(\frac{t^l}{s^k}a\Big)\Big),F_t(u)\bigg]\cdot
{\vphantom{\bigg[}}^{F_t(u)}
\bigg[F_s\Big(x_{\al}\Big(\frac{t^l}{s^k}a\Big)\Big),F_s(h)\bigg]. \nonumber
\end{eqnarray}
\par
Now, for all cases apart from the case, where $G(\Phi,R)=G_{\ad}(\C_l,R)$,
and $\a$ is a long root, by Lemmas~\ref{lem-weights} or~\ref{lem-diagonal}
one can choose a decomposition $F_t(g)=uh$, where $h$ commutes with
$x_{\a}(*)$. Therefore,
$$ F_t(x)=\bigg[F_t\Big(x_{\al}\Big(\frac{t^l}{s^k}a\Big)\Big),
F_t(u)\bigg]. $$
\noindent
Now, by Lemma~\ref{lkjh} one can choose such $l$ and $n$ that
\begin{equation}\label{goal}
F_t(x)\in \big[E(\Phi,F_{st}(s^pt^qR),F_{st}(s^pt^q\ma)),
E(\Phi,F_{st}(s^pt^qR),F_{st}(s^pt^q\mb))\big],
\end{equation}
\noindent
considered as a subgroup of $G(\Phi,R_{st})$.
In general, this is the first factor of the above expression
for $F_t(x)$.
\par
In the exceptional case we can choose $h=h_{\varpi_l}(\e)$,
for some $\e\equiv 1\pmod{s^r\mb}$. Clearly, also
$\e^{-1}\equiv 1\pmod{s^r\mb}$, and thus
$$ \bigg[F_s\Big(x_{\al}\Big(\frac{t^l}{s^k}a\Big)\Big),F_s(h)\bigg]=
x_{\al}\big({t^l}{s^{r-k}}F_{st}(ab)\big), $$
\noindent
for some $b\in\mb$. Now a reference to Lemmas~17 and~\ref{lem3}
shows that one can choose such $l$ and $n$ that the second factor
$$ {\vphantom{\bigg[}}^{F_t(u)}
\bigg[F_s\Big(x_{\al}\Big(\frac{t^l}{s^k}a\Big)\Big),F_s(h)\bigg]=
{\vphantom{\big[}}^{F_t(u)}x_{\al}\big({t^l}{s^{r-k}}F_{st}(ab)\big) $$
\noindent
sits in the same commutator subgroup, as the first factor.
Thus, in all cases we get inclusion (\ref{goal}). In other
words, $F_t(x)$ can be expressed as
$$ F_t(x)=\prod_{i=1}^L
\Big[{\vphantom{\big(}}^{x_{-\b_i}(F_{st}(s^pt^qc_i))}
x_{\b_i}\big(F_{st}(s^pt^qa_i)\big),
{\vphantom{\big(}}^{x_{-\g_i}(F_{st}(s^pt^qd_))}
x_{\g_i}(F_{st}\big(s^pt^qb_i)\big)\Big], $$
\noindent
for some $\b_i,\g_i\in\Phi$, some $a_i\in\ma$, $b_i\in\mb$ and
some $c_i,d_i\in R$.
\par
Form the following product of commutators in $G(\Phi,R_s)$,
$$ y=\prod_{i=1}^L
\Big[{\vphantom{\big(}}^{x_{-\b_i}(F_{s}(s^pt^qc_i))}
x_{\\b_ia}\big(F_{s}(s^pt^qa_i)\big),
{\vphantom{\big(}}^{x_{-\g_i}(F_{s}(s^pt^qd_i))}
x_{\g_i}(F_{s}\big(s^pt^qb_i)\big)\Big], $$
\noindent
by the very construction, $F_t(x)=F_y(y)$. On the other hand,
$x,y\in G(\Phi,R_s,t^nR_s)$ and the restriction of $F_t$
to $G(\Phi,R_s,t^nR_s)$ is injective by Lemma~\ref{noetherian},
it follows $x=y$ and thus we established (\ref{bnm}).
\end{proof}
Now we are in a position to finish the proof of Theorem 1 and,
in fact, of the following much stronger result. Morally, it
culminates all calculations of Sections 7, 8 and 9, and asserts
that for any elements $g_1,\ldots,g_K\in E(\Phi,R_s,\ma_s)$,
in finite number, and any $s$-adic neighborhoods $Y$ and $Z$
of $e$ in the {\it elementary\/} subgroups $E(\Phi,R,\ma)$ and
$E(\Phi,R,\mb)$, respectively, there exists a small $s$-adic
neighbourhood $X$ of $e$ in the principal congruence subgroup
$G(\Phi,R,\mb)$ such that $[g_i,F_s(X)]\subseteq F_s([Y,Z])$,
for all $i$. This is a very powerful result, which will be used
in this form in the proposed description of some classes of
intermediate subgroups. See \cite{VP2,SVY} for a clarification,
why one needs commutator formulae in this stronger form.
In turn, this result can be easily deduced from Lemma~\ref{justlem}
by a standard patching argument using partitions of $1$.
\begin{The}\label{mainlem}
Let\/ $\Phi$ be a reduced irreducible root system,\/ $\rk(\Phi)\ge 2$.
In the cases\/ $\Phi=\C_2,\G_2$ assume additionally that\/ $2\in R^*$.
Then for any\/ $s\in R$, $s\neq 0$, any\/ $p,k$ and\/ $L$, there
exists an\/ $r$ such that for any two ideals\/ $\ma$ and\/ $\mb$
of a commutative ring\/ $R$, one has
\begin{multline}\label{eq:final}
\bigg[E^L\Big(\Phi,\frac{1}{s^k}R,\frac{1}{s^k}\ma\Big),
F_s\big(G(\Phi,R,s^r\mb)\big)\bigg]\subseteq\\
\Big[E\big(\Phi,F_s(s^pR),F_s(s^p\ma)\big),
E\big(\Phi,F_s(s^pR),F_s(s^p\mb)\big)\Big].
\end{multline}
\end{The}
\begin{proof} First we claim that for the same $k$ and $L$ and
{\it any\/} $q$ there exists an $r$ such that
\begin{multline}\label{popt}
\bigg[E^1\Big(\Phi,\frac{1}{s^k}\ma\Big),
F_s\big(G(\Phi,R,s^r\mb)\big)\bigg]\subseteq\\
\Big[E\big(\Phi,F_s(s^qR),F_s(s^q\ma)\big),
E\big(\Phi,F_s(s^qR),F_s(s^q\mb)\big)\Big].
\end{multline}
\noindent
Indeed, let
$\displaystyle{x_\alpha\Big(\frac{1}{s^k}a\Big)\in
E^1\Big(\Phi,\frac{1}{s^k}\ma\Big)}$,
and $g\in G(\Phi,R,s^r\mb)$. For any maximal
ideal $\mm\lhd R$, choose an $t_\mm\in R\backslash\mm$ and a
positive integer $l_\mm$ according to (\ref{bnm}). Since the
collection of all $t_\mm^{l_\mm}$ is not contained in any maximal
ideal, we may find a finite number of them,
$t_1^{l_1},\ldots,t_K^{l_K}$ and such $c_1,\ldots,c_K\in R$ that
$$ t_1^{l_1}c_1+\ldots+t_K^{l_K}c_K=1. $$
\noindent
It follows that
$$ x_\alpha\Big(\frac{1}{s^k}a\Big)=
x_\alpha\bigg(a\sum_{i=1}^K\frac{t_i^{l_i}}{s^k}c_i\bigg)=
\prod_{i=1}^Kx_\alpha\Big(\frac{t_i^{l_i}}{s^k}c_i a\Big). $$
\noindent
Since there are only finitely many factors, it follows from
(\ref{bnm}) that for {\it any\/} $h$ there exists an $r$
such that
\begin{equation}\label{kik}
\bigg[x_\alpha\Big(\frac{t_i^{l_i}}{s^k}c_i a\Big),F_s(g)\bigg]\in
\Big[E\big(\Phi,F_s(s^hR),F_s(s^h\ma)\big),
E\big(\Phi,F_s(s^hR),F_s(s^h\mb)\big)\Big].
\end{equation}
\noindent
A direct computation using (\ref{kik}), Formula (C2) and
Lemma~\ref{lat-bis}, shows that if $h$ was large enough, we get
\begin{multline*}
\bigg[x_\alpha\Big(\frac{1}{s^k}a\Big),F_s(g)\bigg]=
\bigg[\prod_{i=1}^Kx_\alpha\Big(\frac{t_i^{l_i}}{s^k}c_i a\Big),F_s(g)\bigg]\in\\
\Big[E\big(\Phi,F_s(s^qR),F_s(s^q\ma)\big),
E\big(\Phi,F_s(s^qR),F_s(s^q\mb)\big)\Big].
\end{multline*}
\noindent
This proves our claim.
\par
Now, applying to (8) the commutator formula (C2), we see that
if $q$ was large enough, we get
\begin{multline*}
\bigg[E^L\Big(\Phi,\frac{1}{s^k}R,\frac{1}{s^k}\ma\Big),
F_s\big(G(\Phi,R,s^r\mb)\big)\bigg]\subseteq\\
{\vphantom{\Big[}}^{E^{L-1}\Big(\Phi,\frac{1}{s^k}R,\frac{1}{s^k}\ma\Big)}
\Big[E\big(\Phi,F_s(s^pR),F_s(s^p\ma)\big),
E\big(\Phi,F_s(s^pR),F_s(s^p\mb)\big)\Big].
\end{multline*}
\noindent
To finish the proof it only remains to once more invoke
Lemma~\ref{lat-bis}.
\end{proof}
\par
Now we are in a position to prove a slightly weaker statement of Theorem~$\ref{main}$. Namely, 
$$ [E(\Phi,R,\ma),G(\Phi,R,\mb)]=[E(\Phi,R,\ma),E(\Phi,R,\mb)]. $$
To get the inclusion of the left hand side into the right hand side,
set $s=1$ in Theorem~\ref{mainlem}. Inclusion in the other direction is obvious.


\section{Relative versus absolute, and variations}

Using the absolute standard commutator formula and calculations
of Sections~\ref{level2} and~\ref{level} we can give a proof of Theorem~\ref{main} --- but not
of the stronger Theorem~\ref{mainlem}.
\par
\begin{proof}[Proof of Theorem~$\ref{main}$]
By Lemma~\ref{lemma1} one has
$$ \big[E(\Phi,R,\ma),C(\Phi,R,\mb)\big]=
\Big[\big[E(\Phi,R),E(\Phi,R,\ma)\big],C(\Phi,R,\mb)\Big]. $$
\noindent
Since all subgroups here are normal in $G(\Phi,R)$, Lemma~\ref{lemma1}
implies
\begin{multline*}
\big [E(\Phi,R,\ma),C(\Phi,R,\mb)\big]\le \\
 \le\Big [E(\Phi,R,\ma),\big[E(\Phi,R),C(\Phi,R,\mb)\big]\Big]\cdot
\Big [E(\Phi,R),\big[E(\Phi,R,\ma),C(\Phi,R,\mb)\big]\Big].
\end{multline*}
\noindent
Applying the {\it absolute\/} standard commutator formula
\cite[Theorem~1]{HPV} = Lemma~\ref{thm1} above, to the first
factor on the right hand side, we immediately see that it
{\it coincides\/} with $[E(\Phi,R,\ma),E(\Phi,R,\mb)]$.
\par
On the other hand, applying to the second factor on the right hand
Lemma \ref{GUGU} followed by Lemma~\ref{thm1} and Lemma 17, we can
conclude that it is {\it contained\/} in
$$ \big[E(\Phi,R),G(\Phi,R,\ma\mb)\big]=E(\Phi,R,\ma\mb)\le
\big [E(\Phi,R,\ma),E(\Phi,R,\mb)\big]. $$
\noindent
Thus, the left hand side is contained in the right hand side,
the inverse inclusion being obvious.
\end{proof}
Lemma~\ref{lemma1} asserts that the commutator of two elementary subgroups,
one of which is absolute, is itself an elementary subgroup. One can ask,
whether one always has
$$ \big [E(\Phi,R,\ma),E(\Phi,R,\mb)\big]=E(\Phi,R,\ma\mb)? $$
\noindent
Easy examples show that in general this equality may fail quite
spectacularly. In fact, when $\ma=\mb$, one can only conclude that
$$ E(\Phi,R,\ma^2)\le \big [E(\Phi,R,\ma),E(\Phi,R,\ma)\big]\le E(\Phi,R,\ma). $$
\noindent
with right bound attained for some {\it proper\/} ideals, such as an
ideal $\ma$ generated by an idempotent.
\par
Nevertheless, the true reason, why the equality in Lemma~\ref{lemma1} holds, is
not the fact that one of the ideals $\ma$ or $\mb$ coincides with $R$,
but only the fact that $\ma$ and $\mb$ are comaximal.
\begin{The}
Let\/ $\Phi$ be a reduced irreducible root system,\/ $\rk(\Phi)\ge 2$.
When\/ $\Phi=\operatorname{B}_2$ or\/ $\Phi=\operatorname{G}_2$,
assume moreover that\/ $R$ has no residue fields\/ ${\Bbb F}_{\!2}$
of $2$ elements.
Further, let\/ $R$ be a commutative ring and\/ $\ma,\mb\unlhd R$ be
two comaximal ideals of\/ $R$, i.e., $\ma+\mb=R$. Then one has the following equality
$$ \big [E(\Phi,R,\ma),E(\Phi,R,\mb)\big]=E(\Phi,R,\ma\mb). $$
\end{The}
\begin{proof}
First of all, observe that by Lemmas~\ref{lemma1} and~\ref{nllnnn} one has
$$
E(\Phi,R,\ma)=\big[E(\Phi,R,\ma),E(\Phi,R)\big]=
\big [E(\Phi,R,\ma),E(\Phi,R,\ma)\cdot E(\Phi,R,\mb)\big]. $$
\noindent
Thus,
\begin{multline*}
E(\Phi,R,\ma)\le \big[E(\Phi,R,\ma),E(\Phi,R,\ma)\big]\cdot
\big[E(\Phi,R,\ma),E(\Phi,R,\mb)\big]\le\\
\le \big[E(\Phi,R,\ma),E(\Phi,R,\ma)\big]\cdot
E(\Phi,R,\ma\mb).
\end{multline*}
\noindent
Commuting this inclusion with $E(\Phi,R,\mb)$, we see that
\begin{multline*}
\big[E(\Phi,R,\ma),E(\Phi,R,\mb)\big]\le
\Big [\big [E(\Phi,R,\ma),E(\Phi,R,\ma)\big],E(\Phi,R,\mb)\Big]\cdot\\
\big[E(\Phi,R,\ma\mb),E(\Phi,R,\mb)\big].
\end{multline*}
\par
The absolute standard commutator formula, applied to the second
factor, shows that its is contained in
$$ \big[G(\Phi,R,\ma\mb),E(\Phi,R,\mb)\big]\le
\big[G(\Phi,R,\ma\mb),E(\Phi,R)\big]=E(\Phi,R,\ma\mb). $$
\par
On the other hand, applying to the first factor Lemma (C3),
and then again the absolute standard commutator formula, we see
that it is contained in
\begin{multline*}
\Big [\big[E(\Phi,R,\ma),E(\Phi,R,\mb)\big],E(\Phi,R,\ma)\Big]\le \\
\le\big [G(\Phi,R,\ma\mb),E(\Phi,R,\ma)\big]\le\\
\le\big[G(\Phi,R,\ma\mb),E(\Phi,R)\big]=E(\Phi,R,\ma\mb).
\end{multline*}
\par
Together with Lemma~\ref{nllnnn} this finishes the proof.
\end{proof}


\section{Where next?}

In this section we state and very briefly discuss some further
relativisation problems, related to the results of the present paper.
We are convinced that these problems can be successfully addressed
with our methods. Throughout we assume that $\rk(\Phi)\ge 2$.
\par
Outside of some initial observations in Sections 3, 4, and 6, in
the present paper we consider only the usual relative subgroups
depending on one ideal of the ground ring, rather than relative
subgroups defined in terms of admissible pairs. In fact, calculations
necessary to unwind the relative commutator calculus are already
awkward enough with one parameter, especially in rank 2. After some
thought, we decided not to overcharge the first exposition of our
method in this setting with unwieldy technical details. Actually,
most of these details are immaterial for the method itself.
This suggests the following problems.
\begin{prob}
\label{prob--doub} Develop working versions of relative conjugation calculus
and relative commutator calculus, for relative subgroups corresponding
to admissible pairs.
\end{prob}
\begin{prob}
\label{prob-stand} Prove the relative standard commutator formula
$$\big  [E(\Phi,R,\ma,\mc),C(\Phi,R,\mb,\md)\big ]=
\big [E(\Phi,R,\ma,\mc),E(\Phi,R,\mb,\md)\big]. $$
\end{prob}
There is little doubt that what one needs to solve these problems
is a stubborn combination of the methods of the present paper with
those developed by Michael Stein in \cite{Stein71}. Solution of the
following problem is also in sight, and would require mostly
technical efforts.
\begin{prob}
\label{prob-bounds} Obtain explicit length estimates in the
relative conjugation calculus and relative commutator calculus.
\end{prob}
Let us mention some further problems, where we hope to apply
methods of the present paper. Firstly, we have in mind
description of subnormal subgroups of Chevalley groups.
\begin{prob}
\label{p1} Describe subnormal subgroups of a Chevalley
group\/ $G(\Phi,R)$.
\end{prob}
It is well known that this problem is essentially a special case
of the following more general problem.
\begin{prob}
\label{p2} Describe subgroups of a Chevalley group\/ $G(\Phi,R)$,
normalised by the relative elementary subgroup\/ $E(\Phi,R,\mq)$,
for an ideal\/ $\mq\unlhd R$.
\end{prob}
Conjectural answer may be stated as follows: there exists an integer
$m=m(\Phi)$, depending only on $\Phi$, with the following property.
For any subgroup $H\le G(\Phi,R)$ normalised by $E(\Phi,R,\mq)$ there
exist an ideal $\ma\unlhd R$ such that
$$ E(\Phi,R,\mq^m\ma)\le H\le C(\Phi,R,\ma). $$
\noindent
The ideal $\ma$ is unique up to equivalence relation $\lozenge_{\mq}$.
\par
The real challenge is to find the smallest possible value of $m$.
For instance, for the case of $\GL(n,R)$, $n\ge 3$, it has
taken the following values:
\par\smallskip
$\bullet$ $m=7$ for $n\ge 4$, John Wilson, 1972 \cite{wilson72},
\par\smallskip
$\bullet$ $m=24$ (under some stability conditions), Anthony Bak,
1982 \cite{bak82},
\par\smallskip
$\bullet$ $m=6$, Leonid Vaserstein, 1986 \cite{vaser86},
\par\smallskip
$\bullet$ $m=48$, Li Fuan and Liu Mulan, 1987 \cite{LL},
\par\smallskip
$\bullet$ $m=5$, the second
author 1990 \cite{vavilov90},
\par\smallskip
$\bullet$ $m=4$, Vaserstein 1990 \cite{vaser90}.
\par\smallskip\noindent
An exposition of these results with detailed proofs may be found
in \cite{ZZ}. Clearly, \cite{bak82} and \cite{LL} drop out of the
mainstream.
The reason is that \cite{bak82} was published some 15 years after
completion, and \cite{LL} relied upon \cite{bak82}. Nevertheless,
these papers are very pertinent in what concerns discussion of
equivalence relation $\lozenge_{\mq}$.
\par
For other classical groups the best known results are due to
Gerhard Habdank \cite{Ha1,Ha2} and the third author
\cite{ZZ}--\cite{ZZ2}, under assumption $2\in R^*$, and to
You Hong, in general, see the discussion in \cite{RNZ}.
\par
For exceptional groups there are no published results. Recently,
the second and the third authors have modified the third generation
proof of the main structure theorems \cite{VG,VGN}, and obtained
the following values: $m=7$ for Chevalley groups of types $E_6$
and $E_7$. This result will be published in a separate paper.
But to get results with the same bound for groups of type $\E_8$
one will have to use localisation.
\par
Other problems we intend to address with relative concern
description of various classes of intermediate subgroups,
see \cite{NV95,VSsamara,LV} for a survey. In \cite{SVY} we
specifically discuss how localisation comes into play.
Let us mention two of the most immediate such problems.
\begin{prob}
\label{prob-sub} Describe the following classes of subgroups
\par\smallskip
$\bullet$ subgroups in $\GL(27,R)$, containing $E(\E_6,R)$,
\par\smallskip
$\bullet$ subgroups in $\Sp(56,R)$, containing $E(\E_7,R)$.
\end{prob}
These problems are discussed by Alexander Luzgarev in
\cite{luzgarev2004}, where one can find conjectural answers.
Before that the second author and Victor Petrov
\cite{VP1,VP2,VP3,P1,petrov3}, and independently and
simultaneously You Hong \cite{Y1,Y2,Y3} decribed overgroups
of classical groups, in the corresponding $\GL(n,R)$.
The proofs of these results partly relied on localisation.
Immediately thereafter Alexander Luzgarev described subgroups
of $G(\E_6,R)$, containing $E(\F_4,R)$, in his splendid paper
\cite{luzgarev2008}, also using localisation, see also
\cite{Lu-thesis}.
\par
Also, we propose to apply the methods of the present paper
to describe overgroups of subsystem subgroups in exceptional
groups.
\begin{prob}
\label{prob-over} Describe subgroups in  $G(\Phi,R)$, containing
$E(\Delta,R)$, under as\-sump\-tion that $\Delta^\perp=\emptyset$
and all irreducible components of $\Delta$ except maybe one
have rank $\ge 2$.
\end{prob}
The following problem appeared as Problem 9 in \cite{yoga}.
It seems to be extremely challenging, and would certainly require
the full force of localisation-completion\footnote{After the
submission of the present paper, jointly with Alexei Stepanov
we succeeded in solving this problem in the special case of
$\SL(n,R)$. It did in fact require both the full force of the
relative commutator calculus, with two parameters, and new
birelative and trirelative versions of Bak's completion
theorem \cite{B4}, to avoid relativisation with several parameters.
We are positive that the same strategy works
for all Chevalley groups, but it may take quite a while to
supply actual details, primarily because many of the fundamental
results, classical for $\GL(n,R)$, are simply not there in this
larger generality.}.
Its solution would be a simultaneous generalisation
of the results in \cite{RN1,BRN}, as also of our Theorem 1.
\begin{prob}
\label{prob-multi}
Let\/ $R$ be a ring of finite Bass---Serre dimension\/
$\delta(R)=d<\infty$, and let\/ $(I_i,\Gamma_i)$, $1\le i\le m$, be
form ideals of\/ $(R,\Lambda)$. Prove that for any\/ $m>d$ one has
\begin{multline*}
\Big[\big [\ldots[G(\Phi,R,I_1),G(\Phi,R,I_2)],\ldots\big ],G(\Phi,R,I_m)\Big]=\\
\Big[\big [\ldots[E(\Phi,R,I_1),E(\Phi,R,I_2)],\ldots\big ],E(\Phi,R,I_m)\Big].
\end{multline*}
\end{prob}
Let us also reiterate very ambitious Problems 7 and 8 posed
in \cite{RNZ}. The first of these problems refers to the context
of odd unitary groups, as created by Victor Petrov
\cite{P1,petrov2,petrov3}.
\begin{prob}
\label{p7} Generalise results of the present paper to odd
unitary groups.
\end{prob}
One of the first steps towards a solution of this problem, and
other related problems for odd unitary groups was
recently done by Rabeya Basu \cite{Basu}.
\par
The next problem refers to the recent context of isotropic reductive
groups. Of course, it only makes sense over commutative rings, but on the
other hand, a lot of new complications occur, due to the fact that
relative roots do not form a root system, and the interrelations of
the elementary subgroup with the group itself are abstruse even over
fields (the Kneser---Tits problem). Still, we are convinced that after
the recent breakthrough by Victor Petrov and Anastasia Stavrova
\cite{PS08,stavrova} most necessary tools are already there.
See also their subsequent papers with Alexander Luzgarev and
Ekaterina Kulikova \cite{luzgarevstavrova,kulikovastavrova}.
\begin{prob}
\label{p8} Obtain results similar to those of the present paper for
[groups of points of] isotropic reductive groups.
\end{prob}
Of course, here one shall have to develop the whole conjugation and
commutator calculus almost from scratch.
\par
Results of the present paper were first announced in our joint
paper \cite{yoga} with Alexei Stepanov. We thank him for
numerous extremely useful discussions. He and an anonimous referee 
carefully read the original manuscript and suggested many
improvements.



\begin{thebibliography}{30}

\bibitem{abe0}
E.~Abe, Chevalley groups over local rings, {\it T\^ohoku
Math.\ J.\/} {\bf 21} (1969), 474--494.

\bibitem{abe88} E.~Abe, Chevalley groups over commutative rings,
{\it Proc.\ Conf.\ Radical Theory {\rm(}Sendai, 1988{\rm)}\/},
Uchida Rokakuho, Tokyo (1989), 1--23.

\bibitem{abe89} E.~Abe, Normal subgroups of Chevalley
groups over commutative rings, {\it Algebraic $K$-Theory and Algebraic
Number Theory {\rm(}Honolulu, HI, 1987{\rm)}, Contemp.~Math.\/} {\bf 83},
Amer. Math. Soc., Providence, RI (1989), 1--17.

\bibitem{abe91} E.~Abe, Chevalley groups over commutative rings. Normal
subgroups and automorphisms, {\it Second International Conference on
Algebra {\rm(}Barnaul, 1991{\rm)}, Contemp. Math.\/}, {\bf 184},
Amer. Math. Soc., Providence, RI, 1995, 13--23.

\bibitem{abe1}
E.~Abe, K.~Suzuki, On normal subgroups of Chevalley groups
over commutative rings, {\it T\^ohoku Math.\ J.\/} {\bf 28}
(1976), no.1, 185--198.

\bibitem{AS}
H.~Apte, A.~Stepanov, Local-global principle for congruence subgroups
of Chevalley groups, {\tt arXiv:1211.3575v1 [math:RA], 15 Nov 2012}, p.
1--9.

\bibitem{bak82} A.~Bak, Subgroups of the general linear
group normalized by relative elementary subgroup, {\it Algebraic
$K$-theory, Part {\rm II} {\rm(}Oberwolfach, 1980{\rm)}},  1--22.
Lecture Notes Math., Vol. 967, 1982, Springer, Berlin.

\bibitem{B4} A.~Bak, Nonabelian\/ $K$-theory: the nilpotent class
of\/ $K_1$ and general stability, {\it $K$-Theory}
{\bf 4} (1991),  363--397.

\bibitem{BBR} A.~Bak, R.~Basu, R.~A.~Rao, Local-global principle for
transvection groups, accepted in {\it Proc. Amer. Math. Soc.}
{\tt arXiv:0908.3094v2 [math.AC]}.

\bibitem{BRN} A.~Bak, R.~Hazrat and N.~Vavilov, Localization
completion strikes again: relative\/ $K_1$ is nilpotent by abelian.
{\it J. Pure Appl. Algebra} {\bf 213} (2009),  1075--1085.

\bibitem{BS} A.~Bak, A.~Stepanov,
Dimension theory and nonstable\/ $K$-theory for net groups,
{\it Rend. Sem. Mat. Univ. Padova}, {\bf 106} (2001),  207--253.

\bibitem{BV2} A.~Bak, N.~Vavilov, Normality for elementary
subgroup functors, {\it Math. Proc. Camb. Philos. Soc.}
{\bf 118} (1995), no.1, 35--47.

\bibitem{BV3} A.~Bak, N.~Vavilov,  Structure of hyperbolic
unitary groups I: elementary subgroups. {\it Algebra Colloquium}
{\bf 7} (2000), no.2, 159--196.

\bibitem{Bass2} H.~Bass, $K$-theory and stable algebra.
{\it Inst. Hautes Etudes Sci.}, Publ. Math., {\bf22} (1964), 5--60.


\bibitem{bassmilnorserre}
H.~Bass, J.~Milnor, J.-P.~Serre,
Solution of the congruence subgroup problem for $\SL_n$
($n\ge 3$) and $\Sp_{2n}$ ($n\ge 2$). Publ. Math. Inst. Hautes
Etudes Sci. 33 (1967), 59--137.

\bibitem{Basu} R.~Basu, Local-global principle for quadratic and
hermitian groups and the nilpotency of $K_1$. (2012), to appear.

\bibitem{BRK} R.~Basu, R.~A.~Rao, R.~Khanna, On Quillen's local global principle,
{\it Commutative algebra and algebraic geometry}, 17--30, {\it Contemp. Math.},
{\bf 390}, Amer. Math. Soc., Providence, RI, 2005.

\bibitem{Borel1970}  A.~Borel, Properties and linear
representations of Chevalley groups, {\it Seminar on algebraic groups
and related finite groups}, Springer-Verlag, Berlin et al.
1970, 1--55.

\bibitem{BV85} Z.~I.~Borewicz, N.~A.~Vavilov,
The distribution of subgroups in the
general linear group over a commutative ring,
{\it Proc.\ Steklov.\ Inst.\ Math.} (1985),  no. 3, 27--46.

\bibitem{carter1972} R.~Carter, {\it Simple groups of Lie type},
Wiley, London et al. 1972.

\bibitem{costakeller1} D.~L.~Costa, G.~E.~Keller, Radix redux:
normal subgroups of symplectic groups. J. reine angew. Math. 427
(1992), 51--105.

\bibitem{costakeller2} D.~L.~Costa, G.~E.~Keller, On the normal
subgroups of $\operatorname{G}_2(A)$, {\it Trans. Amer. Math. Soc.}
{\bf 351} (1999), no.12, 5051--5088.

\bibitem{Ha1} G.~Habdank, {\it Mixed commutator groups in classical groups
and a classification of subgroups of classical groups normalized by
relative elementary groups.} Doktorarbeit Uni. Bielefeld,
1987, 1--71.

\bibitem{Ha2} G.~Habdank, A classification of subgroups of
$\Lambda$-quadratic groups normalized by relative elementary groups,
{\it Adv. Math.} {\bf 110} (1995), 191--233.

\bibitem{HO} A.~J.~Hahn and O.~T.~O'Meara. {\it The Classical Groups
and $K$-Theory}, Springer, 1989.

\bibitem{RH} R.~Hazrat, Dimension theory and nonstable $K_1$ of
quadratic modules, {\it $K$-Theory} {\bf 27} (2002), 293--328.

\bibitem{RH2}  R.~Hazrat, {\it On\/ $K$-theory of
classical-like groups}, Doktorarbeit Uni. Bielefeld, 2002, 1--62.

\bibitem{HPV} R.~Hazrat, V.~Petrov, N.~Vavilov, Relative subgroups
in Chevalley groups, {\it J. $K$-Theory} {\bf 5} (2010), 603--618.

\bibitem{yoga} R.~Hazrat, A.~Stepanov, N.~Vavilov, Z.~Zhang,
The yoga of commutators, {\it J. Math.~Sci.} {\bf 387} (2011), 53--82.

\bibitem{HSVZ} R.~Hazrat, A.~Stepanov, N.~Vavilov, Z.~Zhang,
On the length of commutators in unitary groups (2012), to appear.

\bibitem{RN1} R.~Hazrat, N.~Vavilov, $K_1$ of Chevalley groups are
nilpotent, {\it J. Pure Appl. Algebra} {\bf 179} (2003), 99--116.

\bibitem{RN} R.~Hazrat, N.~Vavilov, Bak's work on the\/ $K$-theory
of rings, with an appendix by Max Karoubi, {\it J. $K$-Theory},
{\bf 4} (2009), 1--65.

\bibitem{RNZ} R.~Hazrat, N.~Vavilov, Z.~Zhang, Relative unitary commutator
calculus, and applications, {\it J. Algebra}, {\bf 343} (2011) 107--137.


\bibitem{RZ} R.~Hazrat, Z.~Zhang, Generalized commutator formulas,
{\it Comm. in Algebra}, {\bf 39} (2011), 1441--1454.

\bibitem{kopeiko} V.~I.~Kopeiko, The stabilization of symplectic groups
over a polynomial ring, {\it Math.\ U.S.S.R. Sbornik}, {\bf 34} (1978)
655--669.

\bibitem{kulikovastavrova} E.~Kulikova, A.~Stavrova, Centralizer of
the elementary subgroup of an isotropic reductive group,
{\tt arXiv:1012.0278v1.math.AG} (2010), 1--7, to appear.

\bibitem{Li1} Li Fuan, The structure of symplectic group over arbitrary
commutative rings. {\it Acta Math. Sinica {\rm(}N.~S.{\rm)}}
{\bf 3} (1987), 3, 247--255.

\bibitem{Li2} Li Fuan, The structure of orthogonal groups over arbitrary
commutative rings. {\it Chinese Ann. Math. Ser.~B}, {\bf 10} (1989), 3,
341--350.

\bibitem{LL} Li Fuan, Liu Mulan, Generalized sandwich theorem,
{\it $K$-Theory} {\bf 1} (1987), 171--184.

\bibitem{LV} Li Shangzhi, N.~Vavilov, Large subgroup classification
project, (2011), 1--45, to appear.

\bibitem{luzgarev2004} A.~Yu.~Luzgarev, On overgroups of\/ $E(E_6,R)$
and\/ $E(E_7,R)$ in their minimal representations,
{\it J. Math. Sci.} {\bf 134} (2006), no.6, 2558--2571.

\bibitem{luzgarev2008} A.~Yu.~Luzgarev, Overgroups of\/ $E(F_4,R)$
in\/ $G(E_6,R)$, {\it St.-Petersburg Math. J.}, {\bf 20} (2009)
no.6, 955--981.

\bibitem{Lu-thesis} Luzgarev, Alexander, Overgroups of exceptional
groups, Doktorarbeit Univ. St.-Petersburg, (2008), 1--106
(in Russian).

\bibitem{luzgarevstavrova} A.~Yu.~Luzgarev, A.~K.~Stavrova,
Elementary subgroup of an isotropic reductive group is perfect
{\it St.~Peters\-burg Math. J.\/} {\bf 24} (2012), N.5, 881--890.

\bibitem{Mason74}  A.~W.~Mason, A note on subgroups of\/ $\GL(n,A)$
which are generated by commutators,
{\it J. London Math. Soc.} {\bf 11} (1974), 509--512.

\bibitem{MAS1} A.~W.~Mason, On subgroup of\/ $\GL(n,A)$ which are generated
by commutators, II.
{\it J. reine angew. Math.} {\bf 322} (1981), 118--135.

\bibitem{MAS2} A.~W.~Mason, A further note on subgroups of $\GL(n,A)$
which are generated by commutators.
{\it Arch. Math.} {\bf 37} (1981)(5) 401--405.

\bibitem{MAS3} A.~W.~Mason and W.W. Stothers, On subgroup of
$\GL(n,A)$ which are generated by commutators.
{\it Invent. Math.} {\bf 23} (1974), 327--346.

\bibitem{matsumoto}  H.~Matsumoto, Sur les sous-groupes
arithm\'etiques des groupes semi-simples d\'eploy\'es,
Ann.~Sci.~\'Ecole Norm.~Sup. ser.~4, 2 (1969) 1--62.

\bibitem{P1} V.~Petrov, Overgroups of unitary groups.  {\it $K$-Theory}
{\bf 29} (2003), 147--174.

\bibitem{petrov2} V.~A.~Petrov, Odd unitary groups, {\it J. Math. Sci.}
{\bf 130} (2003), no. 3, 4752--4766.

\bibitem{petrov3} V.~A.~Petrov, {\it Overgroups of classical groups},
Doktorarbeit Univ. St.-Petersburg 2005, 1--129 (in Russian).

\bibitem{PS08} V.~A.~Petrov, A.~K.~Stavrova,
Elementary subgroups of isotropic reductive groups,
{\it St.~Peters\-burg Math. J.} {\bf 20}  (2008), no. 3, 160--188.

\bibitem{plotkin}
E.~Plotkin, On the stability of the $K_1$-functor for
Chevalley groups of type $E_7$, {\it J.Algebra}, {\bf 210}
(1998), no.1, 67--95.

\bibitem{Qu} D.~Quillen, Projective modules
over polynomial rings {\it Invent. Math.\/} {\bf 36} (1976)
167--171.

\bibitem{stavrova} A.~K.~Stavrova, {\it Structure of isotropic
reductive groups}, Doktorarbeit Univ. St.-Petersburg 2009, 1--158
(in Russian).

\bibitem{stein2} M.~R.~Stein, Generators, relations and
coverings of Chevalley groups over commutative rings, {\it Amer.
J.~Math.\/} {\bf 93} (1971) no.4, 965--1004.

\bibitem{Stein71} M.~R.~Stein, Relativising
functors on rings and algebraic K-theory,
{\it J.~Algebra} {\bf 19} (1971), no. 1, 140--152.

\bibitem{stein}
M.~R.~Stein, Stability theorems for $K_1$, $K_2$ and related
functors modeled on Chevalley groups, {\it Japan.\ J.\
Math.\/}, {\bf 4} (1978), no.1, 77--108.

\bibitem{steinberg67} R.~Steinberg, {\it Lectures on Chevalley groups},
Yale University, 1967.

\bibitem{step} A.~Stepanov, Universal localisation in algebraic
groups. {\tt http://alexei.stepanov.spb.ru/ publicat.html}
(2010) to appear.

\bibitem{SV} A.~Stepanov, N.~Vavilov, Decomposition of transvections:
a theme with variations. {\it $K$-theory}, {\bf 19} (2000), 109--153.

\bibitem{SV10} A.~Stepanov, N.~Vavilov, On the length of commutators
in Chevalley groups, {\it Israel J. Math.} (2011), 1--20, to appear.

\bibitem{SVY} A.~Stepanov, N.~Vavilov, Hong You, Overgroups of semi-simple
subgroups via localisation-completion. (2012), 1--43, to appear.

\bibitem{Sus} A.~A.~Suslin, On the structure of the
special linear group over the ring of polynomials,
{\it Izv. Akad. Nauk SSSR, Ser. Mat.} {\bf 141} (1977) no.2, 235--253.

\bibitem{suskop}
A.~A.~Suslin, V.~I.~Kopeiko, Quadratic modules and orthogonal groups
over polynomial rings, {\it J.\ Sov.\ Math.}, {\bf 20} (1985)
no.6, 2665--2691.

\bibitem{suzuki}
K.~Suzuki, Normality of the elementary subgroups of twisted
Chevalley groups over commutative rings, {\it J.\ Algebra}
{\bf 175} (1995) no.2, 526--536.

\bibitem{taddei}
G.~Taddei, Normalit\'e des groupes \'el\'ementaires dans les
groupes de Chevalley sur un anneau, {\it Contemp.\ Math.\/},
{\bf 55} (II), (1986), 693--710

\bibitem{tits}
J.~Tits, Syst\`emes g\'en\'erateurs de groupes de congruence,
{\it C.\ R.\ Acad.\ Sci.\ Paris, S\'er A\/}, {\bf 283},
(1976), 693--695

\bibitem{V1} L.~N.~Vaserstein, On the normal subgroups of
$\GL_{n}$ over a ring. {\it Lecture Notes in Math.} {\bf 854}
(1981), 456--465.

\bibitem{vaser86} L.~N.~Vaserstein, The subnormal structure of general linear
groups, {\it Math. Proc. Cambridge Philos. Soc.} {\bf 99} (1986), no.3, 425--431.

\bibitem{vaser2} L.~N.~Vaserstein, On normal subgroups of
Chevalley groups over commutative rings, {\it T\^ohoku Math. J.\/}
{\bf 36} (1986) no.5, 219--230.

\bibitem{vaser90} L.~N.~Vaserstein, The subnormal structure
of general linear groups over rings, {\it Math. Proc. Camb. Phil. Soc.}
{\bf 108} (1990), no.2, 219--229.

\bibitem{VY} L.~N.~Vaserstein, You Hong, Normal subgroups of
classical groups over rings. {\it J. Pure Appl. Algebra.}
{\bf 105} (1995), 93--105.

\bibitem{vavilov90} N.~Vavilov, Subnormal structure of general linear group,
{\it Math. Proc. Camb. Phil. Soc.} {\bf 107} (1990), 103--106.

\bibitem{NV91} N.~Vavilov, Structure of Chevalley groups over
commutative rings, {\it Proc.\ Conf.\ Nonasso\-ciative Algebras
and Related Topics {\rm(}Hiroshima, 1990{\rm)}},
World Sci.\ Publ., London et al., 1991, 219--335.

\bibitem{NV95} N.~Vavilov, Intermediate subgroups in Chevalley groups,
{\it Proc. Conf. Groups of Lie Type and their Geometries
{\rm(}Como -- 1993{\rm)}}, Cambridge Univ.\ Press, 1995, 233--280.

\bibitem{weight}
N.~Vavilov, A third look at weight diagrams, {\it Rend.\ Sem.\
Math.\ Univ. Padova}, {\bf 104} (2), (2000), 1--50.

\bibitem{vavilov07} N.~Vavilov, An $\operatorname{A}_3$-proof of
structure theorems for Chevalley groups of types $\operatorname{E}_6$ and
$\operatorname{E}_7$, I, II. {\it Int. J. Algebra Comput.\/} {\bf 17}
(2007), no.5--6, 1283--1298; {\it St.~Petersburg J. Math.} {\bf 23}
(2011), no.6.

\bibitem{weight-elem} N.~A.~Vavilov, Weight elements of Chevalley
groups {\it St.-Petersburg Math. J.}, {\bf 20} (2009), no.1, 23--57.

\bibitem{VG} N.~A.~Vavilov, M.~R.~Gavrilovich,
An\/ $\mathrm{A}_2$-proof of the structure theorems for Chevalley groups of
types\/ $\mathrm{E}_6$ and\/ $\mathrm{E}_7$, {\it St.-Petersburg Math. J.}, {\bf 16}
(2005), no.4, 649--672.

\bibitem{VGN} N.~A.~Vavilov, M.~R.~Gavrilovich, S.~I.~Nikolenko,
Structure of Chevalley groups: the Proof from the Book,
{\it J. Math. Sci.}, {\bf 140} (2007), no.5, 626--645.

\bibitem{VN} N.~A.~Vavilov, S.~I.~Nikolenko,
An $\mathrm{A}_2$-proof of the structure theorems for Chevalley groups of
types $\mathrm{F}_4$, {\it St.-Petersburg Math. J.} {\bf 20}
(2009), no.4, 27--63.

\bibitem{VP1} N.~A.~Vavilov, V.~A.~Petrov, On overgroups
of\/ $\EO(2l,R)$, {\it J. Math. Sci.}
 {\bf 116} (2003), no.1, 2917--2925.

\bibitem{VP2} N.~A.~Vavilov, V.~A.~Petrov, On overgroups
of\/ $\Ep(2l,R)$, {\it St. Petersburg Math. J.}
 15 (2004), no.4, 515--543.

\bibitem{VP3} N.~A.~Vavilov, V.~A.~Petrov, On overgroups
of\/ $\EO(n,R)$, {\it St. Petersburg Math. J.}
 19 (2008), no.2, 167--195.

\bibitem{vavplot} N.~Vavilov, E.~Plotkin, Chevalley groups over commutative
rings I: Elementary calculations, {\it Acta Applic.\ Math.\/},
{\bf 45}, (1996), 73--113.

\bibitem{VS8}  N.~A.~Vavilov, A.~V.~Stepanov, Standard commutator formula.
{\it Vestnik St.-Petersburg Univ., ser.1},
{\bf 41} (2008), no.1,  5--8.

\bibitem{VSsamara} N.~A.~Vavilov, A.~V.~Stepanov, Overgroups of semisimple
subgroups, {\it Vestnik Samara Univ. Nat. Sci.} (2008), no.3 (62), 51--94,
(in Russian).

\bibitem{VS10} N.~A.~Vavilov, A.~V.~Stepanov, Standard commutator formulae,
revisited, {\it Vestnik St.-Petersburg State Univ., ser.1}, {\bf 43}
(2010), no.1, 12--17.

\bibitem{wilson72} J.~S.~Wilson, The normal and subnormal structure
of general linear groups, Proc. Camb. Phil. Soc. 71 (1972),
163--177.

\bibitem{Y1} You Hong, Overgroups of symplectic group in
linear group over commutative rings, {\it J. Algebra}
{\bf 282} (2004), no. 1, 23--32.

\bibitem{Y2} You Hong, Overgroups of classical groups in
linear group over Banach algebras, {\it J. Algebra} {\bf 304} (2006),
1004--1013.

\bibitem{Y3} You Hong, Overgroups of classical
groups over commutative group in linear group, {\it Sci. China, Ser.~A}
{\bf 49} (2006), no.5, 626--638.

\bibitem{youhong} You Hong, Subgroups of classical groups
normalised by relative elementary groups,
{\it J. Pure Appl. Algebra} {\bf 216} (2011), 1040--1051

\bibitem{ZZ} Z.~Zhang, {\it Lower\/ $K$-theory of unitary groups},
Doktorarbeit Univ. Belfast, 2007, 1--67.

\bibitem{ZZ1} Z.~Zhang, Stable sandwich classification theorem for
classical-like groups. \textit{Proc. Camb. Phil. Soc.}
\textbf{143} (2007),  607--619.

\bibitem{ZZ2} Z.~Zhang, Subnormal structure of non-stable unitary
groups over rings. \textit{J. Pure Appl. Algebra} \textbf{214}
(2010), 622--628.

\bibitem{you92} You Hong, On subgroups of Chevalley groups which
are generated by commutators. {\it J. North\-east Normal Univ.},
2 (1992), 9--13.

\end{thebibliography}
\end{document}